\documentclass[11pt, reqno]{amsart}
\usepackage{a4wide}

\usepackage{amsfonts}
\usepackage{amsmath}
\usepackage{mathrsfs}
\usepackage{stmaryrd}
\usepackage{float}
\usepackage{nicefrac}

\usepackage{graphicx}
\usepackage{bm}
\usepackage{stmaryrd}
\SetSymbolFont{stmry}{bold}{U}{stmry}{m}{n}
\usepackage{graphicx,overpic,microtype,rotating}
\usepackage{microtype}
\usepackage{pgfplots}
\usepackage{algorithm}
\usepackage{algpseudocode}

\usepackage{arxiv_macros}

\definecolor{col1}{rgb}{0.00,0.45,0.74} 
\definecolor{col2}{rgb}{0.85,0.33,0.10} 
\definecolor{col3}{rgb}{0.93,0.69,0.13} 
\definecolor{col4}{rgb}{0.49,0.18,0.56} 
\definecolor{col5}{rgb}{0.47,0.67,0.19} 
\definecolor{col6}{rgb}{0.30,0.75,0.93} 

\pgfplotsset{compat=1.17}

\begin{document}

\title[High-Order Space-Angle-Energy DGFEMs for Boltzmann Transport]{Efficient High-Order Space-Angle-Energy Polytopic Discontinuous Galerkin Finite Element Methods for Linear Boltzmann Transport}

\author[P. Houston \and M. E. Hubbard \and T. J. Radley \and O. J. Sutton \and R. S. J. Widdowson]{Paul Houston \and Matthew E. Hubbard \and Thomas J. Radley \and Oliver J. Sutton \and Richard S.J. Widdowson}

\address[P. Houston]{
School of Mathematical Sciences, University of Nottingham,
University Park, Nottingham NG7 2RD, UK}
\email{Paul.Houston@nottingham.ac.uk}
\address[M. E. Hubbard]{School of Mathematical Sciences, University of Nottingham,
University Park, Nottingham NG7 2RD, UK }
\email{Matthew.Hubbard@nottingham.ac.uk}
\address[T. J. Radley]{School of Mathematical Sciences, University of Nottingham,
University Park, Nottingham NG7 2RD, UK }
\email{Thomas.Radley@nottingham.ac.uk}
\address[O. J. Sutton]{Department of Mathematics,
King's College London,
London,
WC2R 2LS}
\email{Oliver.Sutton@kcl.ac.uk}
\address[R. S. J. Widdowson]{School of Mathematical Sciences, University of Nottingham,
University Park, Nottingham NG7 2RD, UK}
\email{Richard.Widdowson@nottingham.ac.uk}

\maketitle

\begin{abstract}
We introduce an $hp$--version discontinuous Galerkin finite element method (DGFEM) for the linear Boltzmann transport problem. 
A key feature of this new method is that, while offering arbitrary order convergence rates, it may be implemented in an almost identical form to standard multigroup discrete ordinates methods, meaning that solutions can be computed efficiently with high accuracy and in parallel within existing software.
This method provides a unified discretisation of the space, angle, and energy domains of the underlying integro-differential equation and naturally incorporates both local mesh and local polynomial degree variation within each of these computational domains.
Moreover, general polytopic elements can be handled by the method, enabling efficient discretisations of problems posed on complicated spatial geometries. 
We study the stability and $hp$--version {\em a priori} error analysis of the proposed method, by deriving suitable $hp$--approximation estimates together with a novel inf-sup bound.
Numerical experiments highlighting the performance of the method for both polyenergetic and monoenergetic problems are presented.


\end{abstract}

\section{Introduction}

The linear Boltzmann transport problem describes the flow of particles through a scattering and absorbing medium, and is a widely used model in areas as diverse as medical imaging, radiotherapy treatment planning, and the design of nuclear reactors, for example.
Here, we consider the numerical approximation of the stationary form of the problem, seeking a solution which is a function of up to six independent variables: $d$, $d=2,3$, spatial variables varying over a domain in $\Re^d$, $(d-1)$ angular variables on the surface of the $d$-dimensional unit sphere $\angledomain$, and an energy variable on the non-negative real line $\Re_{\geq 0}$. 
The high dimensionality of this problem means that it is imperative to develop efficient numerical approximation methods.
Over the years numerous methods have been proposed for this problem, which we shall briefly review below. 

Given the structure of the underlying problem, the space, angle and energy components of the solution are typically discretised separately using a variety of techniques.
Historically, there has largely been a predominant standard approach to energy discretisation known as the \emph{multigroup approximation}; see~\cite[Chapter 2]{Lewis:1984} and the references cited therein.
Essentially, this approach approximates the energy by a piecewise constant function with respect to a finite number of non-overlapping \emph{energy groups}.
A key appeal of this approach is that the numerical solution is computed by sequentially solving a single monoenergetic Boltzmann transport problem (i.e., only depending on the spatial and angular variables) for each energy group.
This is possible because the scattering process is typically structured in such a way that particles only lose energy in each collision with the medium, either by producing secondary particles or depositing energy locally, and hence the solution in a given energy group only depends on the solution in groups at higher energies, cf., also \cite{Hall:2017}.

On the other hand, discretisations of the angular component of the solution have a rich history and numerous numerical schemes have been proposed.
A few classes of such schemes have received particular attention within the literature due to their numerical properties. 
Spherical harmonic approximations are a widely used form of spectral discretisation in angle, constructed utilising a basis of typically high-order smooth spherical harmonic functions defined globally on the sphere; see \cite{Carson_1947,Fletcher:1983,Lewis:1984}.
The emphasis of such schemes is to simplify the implementation of the scattering operator, typically at the expense of a more expensive implementation of the streaming operator.
Such schemes offer a natural variational setting for their analysis, but the global nature of the basis functions makes local adaptivity a challenging task and Gibbs'-type oscillations may be expected around sharp variations in the solution. 

An alternative strain of methods are collectively known as discrete ordinates methods, in which the angular component of the problem is discretised via collocation at a discrete set of angular quadrature points.
The advantage of this approach is that, when combined with an appropriate linear solver, the Boltzmann transport problem may be solved in parallel as a set of independent linear transport problems in the spatial domain with fixed wind directions.
There appear to be two predominant flavours of discrete ordinates-type methods in the literature, which may be coarsely classified as \emph{global high-order methods} and \emph{local low-order methods}.
Schemes in the former category typically fall within the family of spectral collocation methods, based on sets of interpolatory or quadrature points for high-order spherical harmonic functions on the sphere, designed according to the principles laid out by Sobolev and Vaskevich in \cite{Sobolev:1997}.
Such schemes include those based on the widely used 
level symmetric quadrature formulae in \cite{Lathrop:1965,Carlson:1970,Koch:2004}, 
Lebedev quadrature schemes in \cite{Lebedev:1976,Lebedev:1975},
general double cyclic triangle quadratures in \cite{Koch:1995,Koch:2004}, 
or sets of points arranged on spherical $t$-designs in \cite{Ahrens:2015}, to name but a few.
The appeal of such methods is that they formally approximate the solution using high-order spherical harmonics,
although generating efficient point sets can become difficult for very high-orders, limiting the theoretical accuracy of such schemes.
Moreover, it is typically challenging to produce such point sets adaptively, i.e., to focus quadrature points in zones of the angular domain where higher resolution is required, for instance, around beams or other localised structures present in the underlying solution. 

Complementing these are methods based on quadrature sets constructed locally using an angular mesh.
Typically, the quadrature schemes used are exact for constant functions on each element, such as so-called $T_N$ schemes, cf. \cite{Thurgood:1995}, or sometimes linear or quadratic functions; see \cite{Jarrell:2010,Jarrell:2011,Lau:2016,Lau:2017,Yang:2018}.
In a similar category, we include methods based on interpolation using continuous finite element basis functions in angle, such as those of \cite{Gao:2009}, and schemes incorporating piecewise spherical harmonic approximations on an angular mesh in \cite{Kophazi:2015} and wavelet-based approaches in \cite{Buchan:2005,Adigun:2018}.
Although such schemes formally approximate the solution using lower-order polynomials, the ability to generate a higher fidelity approximation by refining the mesh, either locally or globally, has contributed to their significant popularity. 
Recent work has generalised these schemes to use higher order polynomials in angle in various different ways; see, for example, \cite{Kophazi:2015,Hall:2017,Yang:2018}. 
While such schemes offer the possibility of high-order convergence and mesh adaptivity, underpinned by a variational framework, they can be more challenging to implement efficiently because the high-order nature of the basis functions on each angular element means that the problem does not immediately facilitate a discrete ordinates-like decoupling into independent spatial transport problems.

In this article, we propose a state of the art $hp$-version discontinuous Galerkin finite element method (DGFEM) for the discretisation of the linear Boltzmann transport problem, in which the space, angle, and energy components of the solution are approximated in a unified manner. In many applications, particularly those arising in medical physics, the spatial domain may be highly complicated; to deal with such strong complexity of the physical geometry, in an efficient manner, we admit the use of general polytopic meshes; see, for example, \cite{cangiani2013hp,cangiani2015hp,CaDoGeHo2017} and the references cited therein. The key properties and advantages of the proposed methodology include: the exploitation of a unified DGFEM discretisation of the linear Boltzmann problem over the entire computational domain ensures that the resulting scheme is naturally high-order; note that, in particular, the use of the aforementioned multigroup approximation limits the accuracy of the
 resulting numerical method to first-order. Taking advantage of the intrinsic variational formulation of the scheme means that the convergence and stability analysis of the underlying DGFEM can be developed, which is the key objective of this article. Furthermore, the proposed framework naturally lends itself to the exploitation of $hp$–adaptivity techniques coupled with rigorous a posteriori error estimation to ensure that the spatial, angular, and energy meshes can be focused around solution features of interest. Moreover, as already highlighted above, complex geometries can be efficiently meshed and easily handled. Finally, and perhaps most importantly from a practical viewpoint, the proposed method enables arbitrary order mesh-based approximations to be built independently in each of the space, angle and energy domains, while still being implemented in the same way as conventional multigroup discrete ordinates schemes.
This highly efficient and naturally parallelisable implementation is made possible by exploiting a novel set of basis functions for the polynomial function spaces which satisfy a Lagrangian property at the nodes of a (tensor product) Gaussian quadrature scheme.
We point out that the mathematical convergence results presented in this article complement those of Johnson and Pitk\"aranta in \cite{Johnson:1983}, who derived the first {\em a priori} error estimates for a discrete ordinates DGFEM approximation of the monoenergetic Boltzmann transport problem, albeit under very low regularity assumptions on the analytical solution.

We remark that a very popular alternative computational framework for simulating the linear Boltzmann transport problem are Monte Carlo methods, which are widely used in practice.
Such methods naturally incorporate the stochastic nature of the underlying physical processes and are highly efficient to implement, as the trajectories of individual incoming particles are simulated independently, yet the mean observed behaviour may only be expected to converge with the square root of the number of samples used.
For this reason, B\"orgers in \cite{Borgers:1998} identified that finite element-based methods could expect to perform more efficiently than Monte Carlo-based methods if high-order finite element methods could be utilised in a suitably efficient way. 
Our work therefore provides a stepping stone to answering the open question of how to achieve this objective in practice, as we are able to compute high-order numerical approximations with minimal additional computational overhead compared to conventional multigroup discrete ordinates methods.

The outline for this paper is as follows. In Section~\ref{sec:model} we introduce the linear Boltzmann transport problem. Then in Section~\ref{sec:scheme}, we formulate the unified $hp$--version DGFEM discretisation. 
Section~\ref{Sec:Inverse_and_approx_estimates} introduces the necessary inverse and $hp$--approximation results; on the basis of these bounds, the stability and convergence analysis of the underlying DGFEM is undertaken in Section~\ref{sec:stability}. In Section~\ref{sec:implementation} we outline how the proposed DGFEM may be implemented in a highly efficient and parallelisable manner based on employing a careful selection of the quadrature and local polynomial bases in the angular and energy domains.
The practical performance of the method is assessed in Section~\ref{sec:numerics} through a series of numerical examples. Finally, in Section~\ref{sec:conclusion} we summarise the work presented in this paper and draw some conclusions.

\subsection{Notation}
 
For a bounded open set $\omega \subset \Re^{\spacedim}$, $\spacedim \geq 1$, we write $H^k(\omega)$ to denote the usual Hilbertian Sobolev space of index $k\geq 0$ of real-valued functions defined on $\omega$, endowed with the seminorm $|\cdot|_{H^k(\omega)}$ and norm $\|\cdot\|_{H^k(\omega)}$, as detailed in \cite{Adams:2003}, for example. Furthermore, we let $L_p(\omega)$, $p \in [1,\infty]$, be the standard Lebesgue space on $\omega$, equipped with the norm $\|\cdot\|_{L_p(\omega)}$. Similarly, for a bounded $(\spacedim-1)$--dimensional surface $S$ embedded in $\Re^{\spacedim}$, the spaces $H^k(S)$ are defined in an analogous manner, cf. \cite{Dziuk:2013}, for example.

\section{Model problem}\label{sec:model}
Given an open bounded polyhedral spatial domain $\spacedomain \subset \Re^{\spacedim}$ for $\spacedim = 2$ or $3$, let $\domain = \spacedomain \times \angledomain \times \energydomain$, where $\angledomain = \{ \dir \in \Re^{\spacedim} : \abs{\dir}_2 = 1\}$ denotes the surface of the $\spacedim$-dimensional unit sphere and $\energydomain = \{ \energy \in \Re : \energy \geq 0 \}$ is the real half line.

The linear Boltzmann transport problem reads:
find $u : \domain \to \Re$ such that
\begin{align}
	\dir \cdot \nabla_{\x} u(\x, \dir, \energy) + (\absorption(\x, \dir, \energy) + \totalscattering(\x, \dir, \energy)) u(\x, \dir, \energy)
	 &= 
	 \scattering[u](\x, \dir, \energy) \notag \\
	 & \qquad + f(\x, \dir, \energy) \text{ in } \domain,
	 \notag
	\\
	u(\x, \dir, \energy)
	 &= 
	 \bc(\x, \dir, \energy) \text{ on } \Gamma_{\inflow}, 
	 \label{eq:pde}
\end{align}
where $f,\bc,\absorption, \totalscattering : \domain \to \Re$ are given data terms (discussed further below), $\nabla_{\x}$ is the spatial gradient operator,
and $\Gamma_{\inflow} = \{ (\x, \dir, \energy) \in \bar{\domain} : \x \in \partial \spacedomain \text{ and } \dir \cdot \normal < 0\}$ denotes the inflow boundary of $\domain$, where $\normal$ denotes the unit outward normal vector on the boundary $\partial\spacedomain$ of $\spacedomain$.
The action of the \emph{scattering operator} applied to the solution $u$ is denoted by
\begin{align*}
	\scattering[u](\x, \dir, \energy)
	=
	\int_{\energydomain} \int_{\angledomain} 
	\scatterkernel(\x, \otherdir \to \dir, \otherenergy \to \energy) 
	u(\x, \otherdir, \otherenergy) 
	\d \otherdir \d \otherenergy,
\end{align*}
where $\scatterkernel$ is a specified scattering kernel,
and
$\totalscattering(\x, \dir, \energy) = \int_{\energydomain} \int_{\angledomain} \scatterkernel(\x, \dir \to \otherdir, \energy \to \otherenergy) \d \otherdir \d \otherenergy$.

Physically, the model~\eqref{eq:pde} describes the transport of particles through a scattering medium, and is linear due to the key physical assumption that particles are only scattered by interactions with the medium and do not interact with one another.
The solution $u(\x, \dir, \energy)$ represents the fluence of particles with a particular energy $\energy \in \energydomain$, travelling in direction $\dir \in \angledomain$, passing through the point $\x \in \spacedomain$.
The scattering kernel $\scatterkernel(\x, \otherdir \to \dir, \otherenergy \to \energy)$ represents the proportion of particles at position $\x$ with energy $\otherenergy$ travelling in direction $\otherdir$ which transition to direction $\dir$ and energy $\energy$ as a result of an instantaneous collision with the medium.
Conversely, the reaction coefficient $\absorption(\x, \dir, \energy) + \totalscattering(\x, \dir, \energy)$, commonly referred to as the \emph{total scattering cross section}, models loss of particles from the fluence in direction $\dir$ with energy $\energy$ as they are absorbed by the medium ($\absorption$) or scattered into other directions and energies ($\totalscattering$).

We simplify the model slightly by assuming that the medium is angularly isotropic in the sense that $\absorption(\x, \dir, \energy) \equiv \absorption(\x, \energy)$ and the scattering kernel depends only on the cosine of the angle between the initial and final directions; i.e., $\scatterkernel(\x, \otherdir \to \dir, \otherenergy \to \energy) \equiv \scatterkernel(\x, \otherdir \cdot \dir, \otherenergy \to \energy)$.
This has the implication that $\totalscattering(\x, \dir, \energy) \equiv \totalscattering(\x, \energy)$ by symmetry.
Furthermore, we make the (physically reasonable) assumption that $\scatterkernel(\x, \otherdir \cdot \dir, \otherenergy \to \energy) = 0$ for $\otherenergy < \energy$, which states that particles cannot gain energy by scattering off the medium. 
Finally, we assume that $f$ and $g$ are compactly supported functions of energy, and that there exists a constant $\poscondmin$ such that 
\begin{equation} \label{condition_on_alpha}
	\poscond(\x, \dir, \energy) := \absorption(\x, \dir, \energy) + \frac{1}{2}(\totalscattering(\x, \dir, \energy) -\gamma(\x, \dir, \energy)) \geq \poscondmin > 0,
\end{equation}
where $\gamma(\x, \dir, \energy) = \int_{\energydomain} \int_{\angledomain} \scatterkernel(\x, \otherdir \to \dir, \otherenergy \to \energy) \d \otherdir \d \otherenergy$.
For notational simplicity, henceforth we will suppress the dependence of the data terms $\absorption, \totalscattering, f$ and $\bc$ on the independent variables.

\begin{remark}
	In practice the absorption cross section $\absorption$ may be equal to zero; hence, in this setting, condition \eqref{condition_on_alpha} reduces to the requirement that $\totalscattering - \gamma\geq c_0^{\prime}>0$, $c_0^{\prime} = 2 c_0$, or more precisely that the macroscopic scattering cross section related to outgoing directions and energies ($\totalscattering$) is greater than the corresponding quantity related to the incoming directions and energies ($\gamma$). An important scattering model employed in practice for photons is the Klein-Nishina scattering model, discussed in Section~\ref{sec:numerics}; one can show that this model does indeed satisfy \eqref{condition_on_alpha} within a physical range of energies; see \cite{radley_thesis_2023} for details.
\end{remark}

\section{Discrete scheme}\label{sec:scheme}
We discretise the Boltzmann transport problem \eqref{eq:pde} using a DGFEM approach, seeking an approximate solution which is a product of discontinuous piecewise polynomial functions with respect to meshes defined in the spatial, angular, and energy domains separately. 
For this, we introduce the following notation.

\subsection{Spatial discretisation}

Let $\spacemesh$ be a subdivision of the spatial domain $\Omega$ into non-overlapping open polytopic elements $\element$ with diameter $h_{\element}$ such that $\overline{\Omega} = \cup \,\overline{ \element}$. The set of faces in $\spacemesh$ will be denoted by $\spacefaces$, which are defined as the $(d-1)$-dimensional planar facets of the elements $\element$ present in the mesh $\spacemesh$. For $d=3$, we assume that each planar face of an element $\element \in \spacemesh$ can be subdivided into a set of co-planar $(d-1)$-dimensional simplices and we refer to this set as the set of faces, as in \cite{CaDoGeHo2017}. 
Given $\element \in \spacemesh$, we denote by $p_{\element} \geq 0$ the polynomial degree on $\element$, and define the vector ${\bf p} := (p_{\element} :  \element \in \spacemesh)$. 
The spatial finite element space is then defined by
\begin{align*}
	\spacespace &= \{ v \in L_2(\spacedomain) : v|_{\element} \in \mathbb{P}_{p_{\element}}(\element) \text{ for all } \element \in \spacemesh \},
\end{align*}
where $\mathbb{P}_{k}(\element)$ denotes the space of polynomials of total degree $k$ on $\element$.
We denote by $\partial \element$ the union of $(d-1)$--dimensional open faces of 
the element $\element$. Then, for a given direction $\dir\in \angledomain$ the inflow and outflow parts of $\partial \element$ are defined as
\begin{align*}
\inflowelement &=\{ \x\in \partial\element : ~ \dir \cdot \normal_{\element}(\x)<0\} ,  \\
\outflowelement &=\{\x\in \partial\element : ~ \dir \cdot \normal_{\element}(\x) \geq 0\},
\end{align*}
respectively, where $\normal_{\element}(\x)$ denotes the unit outward normal vector to $\partial \element$ 
at $\x \in \partial\element$. 

Given $\element\in\spacemesh$, the trace of a (sufficiently smooth) function $v$ on $\inflowelement$ from $\element$
 is denoted by $v^+_{\element}$. 
Further, if $\inflowelement \backslash \partial\Omega$ is nonempty, then for $\x \in \inflowelement \backslash \partial\Omega$
there exists a unique $\element^\prime \in \spacemesh$ such that $\x \in \outflowelement^\prime$; with this notation, we
denote by $v^-_{\element}$ the trace of $v |_{\element^\prime}$ on $\inflowelement \backslash \partial\Omega$.
Hence the upwind jump of the function $v$ across a face $F \subset \inflowelement \backslash \partial\Omega$ is denoted by
\begin{equation*}
\ujump{v}:=v^+_{\element}-v^-_{\element} .
\end{equation*}
In the remainder of the article we suppress the subscript $\element$, since it will always be clear which element $\element\in\spacemesh$ the quantities $v_{\element}^\pm$ correspond to.

\subsection{Angular discretisation} \label{sec:angulare_discretisation}

A general framework developed for solving partial differential equations on surfaces has been developed in \cite{dgfem_surface_pdes_2015,Demlow2009HigherOrderFE,Dziuk:2013} and the references cited therein. 
Given that 
our particular setting is greatly simplified, we proceed in a slightly different manner. 
Let $\approxangledomain$ to denote a polyhedral surface in $\Re^d$ composed of (closed)
planar faces  $\approxangleelement$ which are assumed to be either simplices (intervals if $d=2$; 
triangles if $d=3$) 
or (affine) quadrilaterals (when $d=3$). We write $\approxanglemesh = \{\approxangleelement\}$ to denote the
associated regular conforming triangulation of $\approxangledomain$, i.e., 
$\angledomain_h = \cup_{\approxangleelement \in \approxanglemesh} \approxangleelement$. 
We now introduce a smooth invertible mapping $\phi_\angledomain: \approxangledomain \rightarrow \angledomain$; 
for example, assuming the surface is star-shaped with respect to the origin, we may simply define $\phi_\angledomain(\tilde{\dir}) = \abs{\tilde{\dir}}_{2}^{-1} \tilde{\dir}$, where $|\cdot |_2$ denotes the $l_2$-norm.
With this notation, we define a mesh of curved surface elements defined on $\angledomain$ by
$$
\anglemesh = \left\{ \angleelement: \angleelement = \phi_{\angledomain}(\approxangleelement ) ~
\forall \approxangleelement \in \approxanglemesh \right\}.
$$
Crucially, we assume that elements $\approxangleelement\in \approxanglemesh$ are mapped to 
$\angleelement \in \anglemesh$, without any significant rescaling. More precisely, we assume that the 
determinant of the inverse of the first fundamental form of the mapping 
$\phi_\angledomain: \approxangledomain \rightarrow \angledomain$ is uniformly bounded from above and below, 
cf. \cite{Dziuk:2013}. Following \cite{dgfem_surface_pdes_2015,Demlow2009HigherOrderFE,Dziuk:2013}, in the 
case when $\approxanglemesh$ is composed of simplices, then $\approxangledomain$ may, for example, be chosen to be a piecewise
linear approximation of $\angledomain$, whereby the elements forming $\approxangledomain$ may be 
constructed so that their vertices lie on $\angledomain$. 
In the case when quadrilateral elements are employed, 
then an initial polyhedral domain $\approxangledomain^\prime$ may be constructed in a similar fashion, though in 
general the resulting element domains will not be affine. 
In this setting,  
we assume there exists $\approxangledomain$ consisting of affine quadrilateral elements, in such a manner 
that the corresponding quadrilateral facets of $\approxangledomain^\prime$ and $\approxangledomain$ may be 
mapped to one another without any significant rescaling. 
We stress that, irrespective of the specific choice of $\approxangledomain$, the assumption on scaling of the Jacobian of the mapping 
$\phi_\angledomain: \approxangledomain \rightarrow \angledomain$ is crucial to ensure that Lemma~\ref{Lem:demlow}
holds, see Section~\ref{Sec:Inverse_and_approx_estimates} below.

Since the surface we are interested in discretising is simply the unit sphere in $\Re^d$, a practical choice for $\approxangledomain$ is the surface of the cube $[-1,1]^d$. 
This leads to the widely used cube-sphere discretisation of the sphere, and enables a particularly simplified implementation of the method.

Let $\refangleelement \subset \Re^{d-1}$ denote the reference element 
(either a simplex or quadrilateral), $\phi_{\angleelement}: \refangleelement \rightarrow \approxangleelement$, which is assumed to be affine, and define 
$F_{\angleelement}: \refangleelement \rightarrow \angleelement$ by 
$F_{\angleelement} = \phi_{\angledomain} \circ \phi_{\angleelement}$.
For each $\angleelement\in \anglemesh$, let $q_{\angleelement}\geq 0$ denote the polynomial degree used on $\angleelement$, and introduce ${\bf q}:=(q_{\angleelement}: \angleelement\in \anglemesh)$. 
 The finite element space defined on the surface of the sphere $\angledomain$ is then given by
 \begin{align*}
	\anglespace = \{v \in L_2(\angledomain) : v|_{\angleelement} = \hat{v} \circ F_{\angleelement}^{-1}, ~\hat{v} \in {\mathcal R}_{q_{\angleelement}}(\refangleelement) \text{ for all } \angleelement \in \anglemesh \},
\end{align*}
where ${\mathcal R}_k(\refangleelement) = \mathbb{P}_{k}(\refangleelement)$ if $\refangleelement$ is a simplex and   ${\mathcal R}_k(\refangleelement) = \mathbb{Q}_{k}(\refangleelement)$ if $\refangleelement$ is a square; here $\mathbb{Q}_{k}(\refangleelement)$ denotes the space of tensor product polynomials on $\refangleelement$ of degree $k$ in each coordinate direction.

\subsection{Energy discretisation}
We first restrict the energy domain to be a finite interval by selecting minimum and maximum energy cutoffs $\energy_{\min}$ and $\energy_{\max}$, respectively. Due to the assumption that the problem data is compactly supported in energy and the assumption on the structure of the scattering kernel, these limits may be chosen so that the analytical solution is compactly supported in energy; with a slight abuse of notation we refer to $\energydomain$ to be this restricted domain $(\energy_{\min},\energy_{\max})$.

Then, for $N_{\energydomain} \geq 1$, let $\energy_{\max} = \energy_{0} > \energy_{1} > \ldots > \energy_{N_{\energydomain}-1} > \energy_{N_{\energydomain}} = E_{\min}$ define a partition of the energy domain of the problem into $N_{\energydomain}$ \emph{energy groups}.
We will refer to the interval $\energyelement = (\energy_{g}, \energy_{g-1})$ as energy group $g$, $g = 1, \dots, N_{\energydomain}$, and define $\energymesh = \{ \energyelement \}_{g=1}^{N_{\energydomain}}$. To each energy group $\energyelement$, $g = 1, \dots, N_{\energydomain}$, we associate a polynomial degree $r_g \geq 0$. Defining ${\bf r} = (r_{\energyelement})_{g = 1}^{N_{\energydomain}}$, we introduce the energy finite element space
\begin{align*}
	\energyspace = \{ v \in L_2(\energy_{\min}, \energy_{\max}) : v|_{\energyelement} \in \mathbb{P}_{r_{\energyelement}}(\energyelement ) \text{ for all } \energyelement \in\energymesh \}.
\end{align*}

\subsection{Discontinuous Galerkin scheme}

Employing the definitions introduced in the previous sections, we define the full space-angle-energy mesh by
$$
\mesh = \spacemesh \times\anglemesh \times \energymesh
=\{\kappa: \kappa = \element \times\angleelement\times \energyelement, ~\element\in\spacemesh, ~\angleelement\in\anglemesh, ~\energyelement\in\energymesh\}.
$$
Over the mesh $\mesh$, we combine the separate function spaces defined above to obtain the discretisation space
\begin{align*}
	\discretespace = \spacespace \otimes \anglespace \otimes \energyspace,
\end{align*}
and, for any $\dir \in \angledomain$, let $\graphspace_{\dir, h} = \{ v \in L_2(\spacedomain) : \dir \cdot \nabla_{\x} v|_{\element} \in L_2(\element) \text{ for all } \element \in \spacemesh \}$ denote the broken spatial graph space. 

We define the upwind transport bilinear form $a_{\dir}^{\energy} : \graphspace_{\dir, h} \times \graphspace_{\dir, h} \to \Re$ as
\begin{align*}
	a_{\dir}^{\energy}(w, v) 
	=& 
	\sum_{\element \in \spacemesh} 
	\int_{\element} (\dir \cdot \nabla_\x w v + (\absorption + \totalscattering) w v) \d \x
	 \\
	& - \!\!\!\sum_{\element \in \spacemesh} 
	\int_{\inflowelement\backslash\partial\Omega} \!\! (\dir \cdot \normal_{\element}) \ujump{w} v^+ \d s \\
	& - \!\!\!
	\sum_{\element \in \spacemesh}
	\int_{\inflowelement\cap\partial\Omega} (\dir \cdot \normal_{\element}) w^+ v^+ \d s,
\end{align*}
and further define the scattering bilinear form $s_{\dir}^{\energy} : L_2(\domain) \times L_2(\spacedomain) \to \Re$ and load linear form $\ell_{\dir}^{\energy} : \graphspace_{\dir, h} \to \Re$, respectively, by
\begin{align*}
	s_{\dir}^{\energy}(w, v) = \int_{\spacedomain} \scattering[w](\x, \dir, \energy) v \d \x, 
\end{align*}
and
\begin{align*}
	\ell_{\dir}^{\energy}(v) 
	= 
	\int_{\spacedomain} f w \d \x 
	- 
	\sum_{\element \in \spacemesh}
	\int_{\inflowelement\cap\partial\Omega}
	(\dir \cdot \normal_{\element}) \bc w \d s.
\end{align*}
Finally, we introduce the DGFEM: find $u_h \in \discretespace$ such that
\begin{align}\label{eq:dgScheme}
	b(u_h, v_h) \equiv a(u_h, v_h) - s(u_h, v_h) = \ell(v_h)
\end{align}
for all $v_h \in \discretespace$, where $a, s : \discretespace \times \discretespace \to \Re$ and $\ell : \discretespace \to \Re$ are given, respectively, by
\begin{align*}
	a(w_h, v_h)
	=
	\int_{\energydomain} \int_{\angledomain} 
	a_{\dir}^{\energy}(w_h, v_h)&
	\d \dir \d \energy,
	\qquad
	s(w_h, v_h)
	=
	\int_{\energydomain} \int_{\angledomain} 
	s_{\dir}^{\energy}(w_h, v_h) 
	\d \dir \d \energy,
	\\&
	\ell(v_h)
	=
	\int_{\energydomain} \int_{\angledomain} 
	\ell_{\dir}^{\energy}(v_h) 
	\d \dir \d \energy.
\end{align*}
We note that this scheme is consistent in the sense that if the analytical solution $u$ to \eqref{eq:pde} is sufficiently smooth then 
\begin{align*}
	b(u, v) = \ell(v)
\end{align*}
for all $v \in \discretespace$.

\section{Inverse inequalities and approximation theory} \label{Sec:Inverse_and_approx_estimates}

In this section, we briefly outline the key technical results required to analyse the DGFEM defined in~\eqref{eq:dgScheme}; for further details, we refer to \cite{cangiani2015hp,CaDoGeHo2017,cangiani2013hp}. 
We first introduce some assumptions on the polytopic spatial mesh $\spacemesh$. 

\begin{assumption} \label{Assumption:mesh_regularity}
The subdivision $\spacemesh$ is shape regular in the sense that
there exists a positive constant $C_{\rm shape}$, independent of the mesh parameters, such that:
\begin{equation*}
\forall \element \in\spacemesh , \quad  \frac{h_{\element}}{\rho_{\element}} \leq C_{\rm shape},
\end{equation*}
with $\rho_{\element}$ denoting the diameter of the largest ball contained in $\element$.
\end{assumption} 

\begin{assumption} \label{Assumption:bounded_number_of_faces} 
There exists a positive constant $C_F$, 
independent of the mesh parameters, such that
$$
\max_{\element \in\spacemesh} 
\left(
\mbox{card} 
      \left\{ F \in \spacefaces : 
              F \subset \partial \element \right\}
\right) 
\leq C_F. 
$$
\end{assumption} 

In order to state the following $hp$-version inverse estimates, proved in \cite{cangiani2015hp,cangiani2013hp}, 
which are sharp with respect to $(d-k)$--dimensional, $k = 1,\ldots, d-1$, element facet degeneration, we
first recall the following definition.

\begin{definition}\label{poly_assumption}
Let $\tilde{\spacemesh}$ denote the subset of elements $\element \in\spacemesh$ which can each be covered by at most $m_{\spacemesh}$ shape-regular simplices $K_i$, $i=1,\dots, m_{\spacemesh}$, and
$$
\dist(\element, \partial K_i)>C_{as}\diam(K_i)/p_{\element}^{2},
\text{ with }
|K_i|\ge c_{as} |\element| 
$$ 
for all $i=1,\dots, m_{\spacemesh}$, for some $m_{\spacemesh}\in\mathbb{N}$ and  $C_{as}, c_{as}>0$, independent of $\element$ and $\spacemesh$, where $p_{\element}$ denotes the polynomial degree associated with element $\element$, $\element \in\spacemesh$.
\end{definition}

Next we recall the following definition from \cite{cangiani2015hp}.
\begin{definition}\label{meshes1}
For each element $\element \in \spacemesh$, let $\mathcal{F}_{\flat}^{\element}$ denote the family of all possible $d$--dimensional simplices
contained in $\element$ and having at least one face in common with $\element$. 
The notation $\kappa_{\flat}^F$ will be used to indicate a simplex belonging to  $\mathcal{F}_{\flat}^{\element}$ and sharing the face $F$ with $\element$.
\end{definition}

With this definition, we introduce the mesh parameter $h_{\element}^\bot$ defined by
\begin{equation}\label{h_orth}
h_{\element}^\bot: = \min_{F\subset \partial \element}  \frac{\sup_{\kappa_{\flat}^F\subset\element}  
|\kappa_{\flat}^F|}{|F|} d  \qquad  \forall\element \in \spacemesh , ~~ d=2,3,
\end{equation}
and note that $h_{\element}^\bot \leq h_{\element}$.
This enables us to recall the following inverse inequality, cf. \cite{CaDoGeHo2017} (equation (5.23)).
\begin{lemma}\label{Lem:Inverse_face_to_ele}
Let $\element\in\spacemesh$, $F\subset \partial\element$ denote one of its faces. Then, for each $v\in\mathbb{P}_p(\element)$,
we have the inverse estimate
\begin{equation}\label{inv_est_gen}
\|v\|_{L_2(F)}^2 \le \Cinvface
\frac{p^2}{h_{\element}^\bot}\|v\|_{L_2(\element )}^2,
\end{equation}
where $\Cinvface$ is a positive constant, which depends on the shape regularity of the covering of 
$\element$, if $\element \in \tilde{\spacemesh}$, but is independent of the discretisation parameters.
\end{lemma}

To state the $H^1-L_2$ trace inequality, we need the following further assumption.

\begin{assumption}\label{Assumption:sub-triangulation}
We assume that every polytopic element  $\element\in\spacemesh\backslash\tilde{\spacemesh}$, admits a sub-triangulation 
into at most $n_{\spacemesh}$  shape-regular simplices $\mathfrak{s}_i$, $i=1,2,\dots , n_{\spacemesh}$, such that
$\bar{\element} = \cup_{i=1}^{n_{\spacemesh}} \bar{\mathfrak{s}}_i$ and
$$
|\mathfrak{s}_i| \geq \hat{c} |\element |
$$
for all $i=1,\dots, n_{\spacemesh}$, for some $n_{\spacemesh}\in\mathbb{N}$ and  $\hat{c}>0$, independent of $\element$ and $\spacemesh$.
\end{assumption}

\begin{lemma} [\cite{CaDoGeHo2017} (Lemma 14)] \label{Lem:Inverse_H1_L2}
Given Assumptions \ref{Assumption:mesh_regularity} and \ref{Assumption:sub-triangulation} are satisfied, for each 
$v\in\mathbb{P}_p(\element)$, $\element \in \spacemesh$,
the inverse estimate 
\begin{equation}
\|\nabla_{\x} v\|_{L_2(\element )}^2 \le \Cinvele \frac{p^4}{h_{\element}^2}\|v\|_{L_2(\element)}^2,
\end{equation}
holds, with constant $\Cinvele$ independent of the element diameter $h_{\element}$,
the polynomial order $p$, and the function $v$, but dependent on the shape regularity of the covering of 
$\element$, if $\element \in \tilde{\spacemesh}$, or  
the sub-triangulation of $\element$, if $\element \in \spacemesh \backslash\tilde{\spacemesh}$.
\end{lemma}

Furthermore we recall the following multiplicative trace inequality, see \cite{cangiani2013hp}, but written in a slightly different form, see \cite{cangiani2015hp}.
\begin{lemma} \label{Lem:trace_inequality}
	For $v\in H^1(\element)$, $\element\in\spacemesh$, given $F\subset\partial\element$, the following bound holds
	\begin{align*}
		\norm{v}^2_{L_2(F)} \leq \frac{C_T}{h_{\element}^\bot} 
		\left(
		\norm{v}^2_{L_2(\element)} + h_{\element} \norm{v}_{L_2(\element)} \norm{\nabla_\x v}_{L_2(\element)}
		\right),
	\end{align*}
	where $C_T$ is a positive constant which is independent of the element diameter $h_{\element}$.
\end{lemma}

We now turn our attention to deriving suitable $hp$--version approximation results on each of the finite element spaces $\spacespace$, $\anglespace$ and $\energyspace$. Starting with the spatial finite element space $\spacespace$, we first introduce the following covering of the mesh $\spacemesh$, see \cite{cangiani2013hp}.

\begin{definition}\label{def:mesh_covering}
A (typically overlapping)
\emph{covering} $\mathcal{T}_{\sharp}  = \{ \mathcal{K} \}$ 
related to the polytopic mesh $\spacemesh$ is a   
set of shape-regular $d$--simplices  $\mathcal{K}$, such that for each $\element\in\spacemesh$, there exists a $\mathcal{K}\in\mathcal{T}_{\sharp}$,
with $\element \subset\mathcal{K}$. Moreover, we assume there exists a covering such that 
$
\diam(\mathcal{K})\le C_{\diam} h_{\kappa},
$
for each pair $\element\in\spacemesh$, $\mathcal{K}\in\mathcal{T}_{\sharp}$, with $\element \subset\mathcal{K}$, for a constant $C_{\diam}>0$, uniformly with respect to the meshsize.
\end{definition}

Furthermore, we introduce the following extension operator from \cite{St1970} (Theorem 5) and \cite{sauter_1996} (Theorem 3).
 
\begin{theorem} \label{thm-extension}
Let $D$ be a domain with mi\-ni\-mally smooth boundary. Then, there exists a linear extension
operator $\mathfrak{E}:H^s(D) \rightarrow H^s({\mathbb R}^d)$, $s \in {\mathbb N}_0\equiv\left\{0,1,2,\ldots\right\}$, such that
$\mathfrak{E}v|_{D}=v$ and
$$
\| \mathfrak{E} v \|_{H^s({\mathbb R}^d)} \leq C_{\mathfrak{E}} \| v \|_{H^s(D)},
$$
where $C_{\mathfrak{E}}$ is a positive constant depending only on $s$ and parameters which characterize the boundary $\partial D$.
\end{theorem}

With this notation we recall the approximation result from \cite{cangiani2013hp} (Theorem~4.2).

\begin{lemma}\label{Lem:spatial_approximation}
Let $\element\in\spacemesh$ and $\mathcal{K}\in\mathcal{T}_{\sharp}$ denote  the corresponding simplex 
such that $\element \subset\mathcal{K}$, cf. Definition~\ref{def:mesh_covering}.
Suppose that $v\in L_2(\Omega)$ is such that $\mathfrak{E} v|_{\mathcal{K}}\in H^{l_{\element}}(\mathcal{K})$, for some $l_{\element}\ge 0$. Then, there exists $\Pi_{\spacedomain} v$, such that
$\Pi_{\spacedomain} v|_{\element} \in \mathbb{P}_{p_{\element}}(\element )$, and the
following bound holds
\begin{equation}\label{approxH_k}
\norm{v - \Pi_{\spacedomain} v}_{H^m(\element )}
\le C \frac{h_{\element}^{s_{\element}-m}}{p_{\element}^{l_{\element}-m}}\norm{ \mathfrak{E}v}_{H^{l_{\element}}(\mathcal{K} )},\quad l_{\element}\ge 0,
\end{equation}
for $0\le m\le l_{\element}$.
Here, $s_{\element}=\min\{p_{\element}+1, l_{\element}\}$ and $C$ is a positive constant, 
that depends on the shape-regularity of ${\mathcal{K}}$, but is
independent of $v$, $h_{\element}$, and $p_{\element}$.
\end{lemma}

A careful inspection of the proof of Theorem~\ref{thm-extension} reveals that the constant 
$C_{\mathfrak{E}}$ is independent of the measure of the underlying domain $D$, cf. \cite{AHPS_2020}.
Hence, employing Theorem~\ref{thm-extension}, the bound \eqref{approxH_k} 
given in
Lemma~\ref{Lem:spatial_approximation} may be stated in the following simplified form:
\begin{equation}\label{approxH_k_2}
\norm{v - \Pi_{\spacedomain} v}_{H^m(\element )} 
\le C \frac{h_{\element}^{s_{\element}-m}}{p_{\element}^{l_{\element}-m}}\norm{ v}_{H^{l_{\element}}(\element )},\quad l_{\element}\ge 0,
\end{equation}
for $0\le m\le l_{\element}$,
and therefore the condition placed on the amount of overlap of the simplices $\mathcal{K}$ in \cite{cangiani2015hp,CaDoGeHo2017,cangiani2013hp} is not required.

To construct a projection operator onto the angular finite element space $\anglespace$, some care 
is required to account for the curvature of $\angledomain$; for completeness we recall the key steps.
Under our assumptions on the mapping 
$\phi_\angledomain: \approxangledomain \rightarrow \angledomain$, 
we first recall the following result from \cite{dgfem_surface_pdes_2015,Demlow2009HigherOrderFE}.

\begin{lemma} \label{Lem:demlow}
Let $v \in H^j (\angleelement)$, $j\geq 0$; then writing $\tilde{v} = v \circ \phi_{\angledomain}$, we have that
\begin{align*}
	\frac{1}{C} \| v \|_{L_2(\angleelement )} \leq & \| \tilde{v} \|_{L_2(\approxangleelement )}
	\leq C \| v \|_{L_2(\angleelement )}, 
	\quad \text{ and } \quad
	|\tilde{v}|_{H^j(\approxangleelement )} \leq C \|v\|_{H^j(\angleelement )},
\end{align*}
where $C$ is a positive constant, which is independent of the meshsize $h_{\angleelement}$.
\end{lemma}

Employing $hp$--approximation results for standard shaped elements, we recall the following result from \cite{Babuska-Suri:RAIRO:1987,schwab}.
\begin{lemma}\label{Lem:approximation} 
Suppose that $\mathfrak{K}$ is a $d$--simplex or $d$--parallelepiped of diameter
$h_{\mathfrak{K}}$. Suppose further that $v|_{\mathfrak{K}} \in H^{l_{\mathfrak{K}}}(\mathfrak{K})$, 
$l_{\mathfrak{K}} \geq 0$.
Then, there exists $\hat{\Pi}_{p_{\mathfrak{K}}}v$
in $\mathcal{R}_{p_{\mathfrak{K}}}(\mathfrak{K})$, $p_{\mathfrak{K}} = 1,2, \dots,$
such that for $0 \leq m \leq l_{\mathfrak{K}}$,
\[ \|v - \hat{\Pi}_{p_{\mathfrak{K}}} v\|_{H^m(\mathfrak{K})}
\leq C \frac{h_{\mathfrak{K}}^{s_{\mathfrak{K}}-m}}{p_{\mathfrak{K}}^{l_{\mathfrak{K}}-m}}
\|v\|_{H^{l_{\mathfrak{K}}}(\mathfrak{K})},
\]
where $s_{\mathfrak{K}} = \min\{p_{\mathfrak{K}}+1, l_{\mathfrak{K}}\}$ and $C$ is a
positive constant, independent of $v$ and the discretisation
parameters.
\end{lemma}

Equipped with Lemma~\ref{Lem:approximation}, we introduce the projection operator 
$\Pi_{\angledomain}$  by
$$
\Pi_{\angledomain} v|_{\angleelement} = (\hat{\Pi}_{q_{\angleelement}} v|_{\angleelement} \circ \phi_{\angledomain}) \circ \phi_{\angledomain}^{-1}
$$
for all $\angleelement\in\anglemesh$.
Hence, employing Lemmas~\ref{Lem:demlow} \& \ref{Lem:approximation}, together with the definition of $\Pi_{\angledomain}$ we deduce the following result.

\begin{lemma} \label{Lem:approximation_angle}
	Let $\angleelement \in\anglemesh$, then given $v|_{\angleelement}\in H^{l_{\angleelement}}(\angleelement )$, for some $l_{\angleelement}\ge 0$, the
following bound holds
\begin{equation*}
\norm{v - \Pi_{\angledomain} v}_{L_2(\angleelement )}
\le C \frac{h_{\angleelement}^{s_{\angleelement}}}{q_{\angleelement}^{l_{\angleelement}}}\norm{v}_{H^{l_{\angleelement}}(\angleelement )},\quad l_{\angleelement }\ge 0,
\end{equation*}
where $s_{\angleelement}=\min\{q_{\angleelement}+1, l_{\angleelement}\}$ and $C$ is a positive constant, that depends on the shape-regularity of $\angleelement$, but is
independent of $v$, $h_{\angleelement}$, and $q_{\angleelement}$.
\end{lemma}

For approximation with respect to energy, we simply define the projection operator $\Pi_{\energydomain}$ by $\Pi_{\energydomain}v |_{\energyelement} = \hat{\Pi}_{r_{\energyelement}} v|_{\energyelement}$, for $g=1,\ldots,N_{\energydomain}$. Collecting these three projection operators, we define $\Pi:L_2(\domain)\rightarrow \discretespace$ by $\Pi = \Pi_{\Omega}\Pi_{\angledomain}\Pi_{\energydomain}$. With this notation we state the following approximation result for the projection operator $\Pi$.

\begin{lemma}\label{lem:projectionApproximation}
	Let $\kappa\in\mesh$ such that $\kappa=\element\times\angleelement\times\energyelement$, $\element\in\spacemesh$, $\angleelement\in\anglemesh$, $\energyelement\in\energymesh$, then given $v|_{\kappa}\in H^{l_\kappa}(\kappa)$, $l_\kappa \geq 0$, the following bound holds
	\begin{align} \label{approx:l2_norm}
		\norm{v-\Pi v}^2_{L_2(\kappa)} \leq C \left( 
		\frac{h_{\element}^{2s_{\element}}}{p_{\element}^{2l_{\kappa}}}
		+\frac{h_{\angleelement}^{2s_{\angleelement}}}{q_{\angleelement}^{2l_{\kappa}}}
		+\frac{h_{\energyelement}^{2s_{\energyelement}}}{r_{\energyelement}^{2l_{\kappa}}}
		\right) \norm{v}_{H^{l_\kappa}(\kappa)}^2.
	\end{align}
	Furthermore, assuming $v|_{\kappa} \in H^{l_\kappa}(\kappa)\cup H^1(\element;H^{l_\kappa}(\angleelement\times \energyelement))$, $l_\kappa\geq 1$, we have that
	\begin{align}
		\norm{\nabla_\x (v-\Pi v)}^2_{L_2(\kappa)} \leq& ~C 
		\left( 
		\frac{h_{\element}^{2s_{\element}-2}}{p_{\element}^{2l_{\kappa}-2}}\norm{v}_{H^{l_\kappa}(\kappa)}^2
		\right. \nonumber \\
		& \left.
		+\left(
		 \frac{h_{\angleelement}^{2s_{\angleelement}}}{q_{\angleelement}^{2l_{\kappa}}}
		+\frac{h_{\energyelement}^{2s_{\energyelement}}}{r_{\energyelement}^{2l_{\kappa}}}
		\right)
		\norm{\nabla_\x v}_{L_2(\element;H^{l_\kappa}(\angleelement\times\energyelement ) )}^2
		\right), \label{approx:h1_space_norm}
\end{align}
and
\begin{align}
		&\int_{\energyelement}\int_{\angleelement} \norm{v-\Pi v}^2_{L_2(\partial \element)} \d \dir \d \energy \nonumber \\
		& ~~~~~~ \leq ~C
		\left(
		\frac{1}{h_{\element}^\perp}\left( 
		\frac{h_{\element}^{2s_{\element}}}{p_{\element}^{2l_{\kappa}-2}}
		+\frac{h_{\angleelement}^{2s_{\angleelement}}}{q_{\angleelement}^{2l_{\kappa}}}
		+\frac{h_{\energyelement}^{2s_{\energyelement}}}{r_{\energyelement}^{2l_{\kappa}}}
		\right) \norm{v}_{H^{l_\kappa}(\kappa)}^2
		\right. \nonumber\\
		& ~~~~~~~~~~
		\left.
		+\frac{h_{\element}^2}{h_{\element}^\perp} \left(
		 \frac{h_{\angleelement}^{2s_{\angleelement}}}{q_{\angleelement}^{2l_{\kappa}}}
		+\frac{h_{\energyelement}^{2s_{\energyelement}}}{r_{\energyelement}^{2l_{\kappa}}}
		\right)
		\norm{\nabla_\x v}_{L_2(\element;H^{l_\kappa}(\angleelement\times\energyelement ) )}^2
		\right). \label{approx:face_terms}
	\end{align}
	Here, $s_{\element} = \min(p_{\element}+1,l_\kappa)$, $s_{\angleelement} = \min(q_{\angleelement}+1,l_\kappa)$, $s_{\energyelement} = \min(r_{\energyelement}+1,l_\kappa)$, and $C$ is a positive constant that depends on the shape regularity of the element $\kappa$, but is independent of the mesh parameters.
\end{lemma}

\begin{proof}
	We start by first writing the projection error in the form
	\begin{align*}
		v-\Pi v 
		= v -\Pi_{\energydomain} v + \Pi_{\energydomain} (v - \Pi_{\angledomain} v) 
		  + \Pi_{\energydomain} \Pi_{\angledomain} (v - \Pi_{\spacedomain} v).
	\end{align*}
	Then \eqref{approx:l2_norm} follows immediately upon application of the triangle inequality, employing the $L_2$-stability of $\Pi_{\energydomain}$ and $\Pi_{\angledomain}$, and the approximation results stated in Lemma~\ref{Lem:spatial_approximation}, cf. \eqref{approxH_k_2}, Lemma~\ref{Lem:approximation} and Lemma~\ref{Lem:approximation_angle}. The proof of \eqref{approx:h1_space_norm} follows in an analogous fashion. To derive \eqref{approx:face_terms}, we first employ the trace inequality stated in Lemma~\ref{Lem:trace_inequality}, together with \eqref{approx:l2_norm} and \eqref{approx:h1_space_norm}.
\end{proof}

\section{Stability and convergence of the discrete scheme}\label{sec:stability}

In this section we study the stability and convergence of the DGFEM~\eqref{eq:dgScheme}. 
To this end, we introduce the DGFEM-\emph{energy norm}
\begin{align}\label{eq:dGnorm}
	\triplenorm{v}^2 
	=& 
	\|\sqrt{\poscond} \, v\|_{L_2(\domain )}^2 \nonumber\\
	&+ 
	\frac{1}{2} 
	\int_{\energydomain}
	\int_{\angledomain} 
	\sum_{\element\in\spacemesh}
	\Big( \| v^+ - v^-\|_{\inflowelement\backslash\partial\Omega}^2
	+
	\| v^+\|_{\partial\element \cap \partial\Omega}^2
	\Big)
	\d \dir
	\d \energy,
\end{align}
and \emph{streamline norm}
\begin{align*}
	\sdnorm{v}^2
	=
	\triplenorm{v}^2
	+
	\int_{\energydomain}
	\int_{\angledomain}
	\sum_{\element \in \spacemesh}
	\tau_{\element}
	\norm{\dir \cdot \nabla_\x v}_{L_2(\element)}^2
	\d \dir
	\d \energy.
\end{align*}
Here, $\|\cdot \|_{\omega}$, $\omega \subset \partial \element$, denotes the (semi)norm associated with the (semi)inner product $(v,w)_\omega = \int_\omega |\dir \cdot \normal_{\element} |vw \d s$. Furthermore, for $\element\in\spacemesh$, we define
$$
\tau_{\element} = \frac{h_{\element}^\bot}{p_{\element}^2}.\
$$

Firstly, we state the following coercivity bound.
\begin{theorem}[Coercivity]\label{thm:coercivity}
	The DGFEM \eqref{eq:dgScheme} is coercive with respect to the DGFEM-energy norm $\triplenorm{\cdot}$ in the sense that the following bound holds:
	\begin{align*}
		b(v, v) \geq \triplenorm{v}^2
	\end{align*}
	for all $v \in \discretespace$.
\end{theorem}
\begin{proof}
	Integrating by parts and rearranging the face terms, the transport bilinear form satisfies
	\begin{align*}
		a_{\dir}^{\energy}(v, v) 
		= 
		\norm{(\absorption + \totalscattering)^{1/2} v}_{L_2(\spacedomain)}^2 
		+ 
		\frac{1}{2} \sum_{\element\in\spacemesh}
	\Big( \| v^+ - v^-\|_{\inflowelement\backslash\partial\Omega}^2
	+
	\| v^+\|_{\partial\element \cap \partial\Omega}^2
	\Big),
	\end{align*}
	as shown in \cite{hss_2002}.
	Recalling that $\totalscattering(\x, \dir,\energy) = \int_{\energydomain} \int_{\angledomain} \scatterkernel(\x, \dir \cdot \otherdir, \energy \to \otherenergy) \d \otherdir \d \otherenergy$ and $\gamma(\x, \dir,\energy) = \int_{\energydomain} \int_{\angledomain} \scatterkernel(\x, \dir \cdot \otherdir, \otherenergy \to \energy) \d \otherdir \d \otherenergy$, employing the Cauchy-Schwarz inequality implies that the scattering term may be bounded by
	\begin{align*}
		s(v, v)
		&=
		\int_{\energydomain} 
		\int_{\angledomain}  
		\int_{\energydomain} 
		\int_{\angledomain} 
		\int_{\spacedomain}
		\scatterkernel(\x, \dir \cdot \otherdir, \otherenergy \to \energy) v(\x, \otherdir, \otherenergy) v(\x, \dir, \energy) 
		\d \x
		\d \otherdir 
		\d \otherenergy
		\d \dir
		\d \energy
		\\&
		\leq
		\|\totalscattering^{\nicefrac{1}{2}} v \|_{L_2(\domain)} \|\gamma^{\nicefrac{1}{2}} v \|_{L_2(\domain)} 
		\leq \frac{1}{2} \|\totalscattering^{\nicefrac{1}{2}} v \|_{L_2(\domain)}^2 
		+ \frac{1}{2} \|\gamma^{\nicefrac{1}{2}} v \|_{L_2(\domain)}^2.
	\end{align*}
	The result then follows by combining these bounds with the definition of $\poscond$ in \eqref{condition_on_alpha}.
\end{proof}

We now derive an inf-sup stability result in the streamline norm $\sdnorm{\cdot}$.

\begin{theorem}[Inf-sup stability]\label{thm:infsup}
Given that Assumptions~\ref{Assumption:mesh_regularity}, \ref{Assumption:bounded_number_of_faces}, and \ref{Assumption:sub-triangulation} hold, then the DGFEM \eqref{eq:dgScheme} is inf-sup stable in the streamline norm, i.e., there exists a constant $\Lambda > 0$, independent of discretisation parameters, such that
	\begin{align*}
		\inf_{v \in \discretespace \setminus \{0\}} \sup_{w \in \discretespace \setminus \{0\}} \frac{b(v, w)}{\sdnorm{v} \sdnorm{w}} \geq \Lambda.
	\end{align*}
\end{theorem}
\begin{proof}
	The proof follows a standard form for inf-sup results, and is similar to the argument presented in \cite{cangiani2015hp} for a scalar advection problem, adapted to the Boltzmann setting.
	To this end, we construct a function $w \in \discretespace$ for each $v \in \discretespace$ such that $\sdnorm{w} \leq \Lambda_1 \sdnorm{v}$ and 
	$b(v, w) \geq \Lambda_2 \sdnorm{v}^2$.
	The result then follows with $\Lambda = \nicefrac{\Lambda_2}{\Lambda_1}$.
	
	Let $w(\x, \dir,E) = v(\x, \dir,E) + \delta v_s(\x, \dir,E)$ where $\delta > 0$ is a constant which will be determined, depending only on the problem data, and $v_s(\x, \dir,E)|_{\element} = \tau_{\element} \dir \cdot \nabla_{\x} v(\x, \dir,E)$ on each spatial element $\element \in \spacemesh$.
	To prove that there exists $C>0$ such that $\sdnorm{w} \leq C \sdnorm{v}$, we apply the triangle inequality to find
	\begin{align*}
		\sdnorm{w} \leq \sdnorm{v} + \delta \sdnorm{v_s},
	\end{align*}
	and bound each term of $\sdnorm{v_s}$ by $\sdnorm{v}$ individually.
	Observing that $\abs{\dir} = 1$, upon application of the inverse inequality stated in Lemma~\ref{Lem:Inverse_H1_L2}, recalling the definition of $\tau_{\element}$ and noting that $h_{\element}^\bot \leq h_{\element}$, we deduce that
	\begin{align*}
		\|\sqrt{\poscond}v_s\|_{L_2(\domain )}^2 
		&=
		\int_{\energydomain}
		\int_{\angledomain}
		\sum_{\element \in \spacemesh}
		\tau_{\element}^2 \norm{\sqrt{\poscond} \dir \cdot \nabla_{\x} v}_{L_2(\element)}^2 
		\d \dir
		\d \energy \\
		&\leq
		\frac{\Cinvele \norm{\poscond}_{L_{\infty}\!(\domain)}}{\poscondmin}
		\norm{\sqrt{\poscond} v}_{L_2(\domain)}^2.
	\end{align*}
	Similarly, we have
	\begin{align*}
		&\int_{\energydomain}
		\int_{\angledomain}
		\sum_{\element \in \spacemesh}
		\tau_{\element}
		\norm{\dir \cdot \nabla_{\x} v_s}_{L_2(\element)}^2 
		\d \dir
		\d \energy \\
		& \qquad \leq
		\Cinvele
		\int_{\energydomain}
		\int_{\angledomain}
		\sum_{\element \in \spacemesh}
		\tau_{\element}
		\norm{\dir \cdot \nabla_{\x} v}_{L_2(\element)}^2
		\d \dir
		\d \energy.
	\end{align*}
	We now consider the face terms arising in the definition of the streamline norm $\sdnorm{\cdot}$. Noting 
	that $\abs{\dir \cdot \normal_{\element}} \leq 1$, applying the inverse inequality stated in Lemma~\ref{Lem:Inverse_face_to_ele} gives
	\begin{align*}
	&\frac{1}{2} 
	\int_{\energydomain}
	\int_{\angledomain} 
	\sum_{\element\in\spacemesh}
	\Big( \| v_s^+ - v_s^-\|_{\inflowelement\backslash\partial\Omega}^2
	+
	\| v_s^+\|_{\partial\element \cap \partial\Omega}^2
	\Big)
	\d \dir
	\d \energy 
	\\
	& \qquad
	\leq \int_{\energydomain}
	\int_{\angledomain} 
	\sum_{\element\in\spacemesh}
	\sum_{F\subset \partial\element}
	\| v_s^+\|_{L_2(F)}^2
	\d \dir
	\d \energy
	\leq \Cinvface C_F \sdnorm{v}^2. 
	\end{align*}
	Since the terms resulting from these bounds are components of $\sdnorm{\cdot}$, it follows that 
	\begin{align*}
		\sdnorm{w} 
		\leq 
		\Lambda_1
		\sdnorm{v}
		\quad\text{ with }\quad
		\Lambda_1 
		= 
		1 + \delta 
		\Big(
		\Cinvele \Big(1+\frac{\norm{\poscond}_{L_{\infty}(\domain)}}{\poscondmin} \Big)
		+ \Cinvface C_F
		\Big)^{1/2}.
	\end{align*}
	
	We now show that $b(v, w) \geq \Lambda_2 \sdnorm{v}^2$.
	By linearity and the coercivity bound stated in Theorem~\ref{thm:coercivity}, we deduce that
	\begin{align} \label{eq:infsup_eq1}
		b(v, w) = b(v, v) + \delta b(v, v_s) 
		\geq 
		\triplenorm{v}^2 + 
		\delta 
		(a(v, v_s) - s(v, v_s)),
	\end{align}
	and expanding the second term on the right-hand side of \eqref{eq:infsup_eq1} gives
	\begin{align*}
		a(v, v_s)
		&=	
		\int_{\energydomain}
		\int_{\angledomain} 
		\sum_{\element \in \spacemesh} 
		\tau_{\element}
		\Big(
		\norm{\dir \cdot \nabla_\x v}_{L_2(\element)}^2
		+
		\int_{\element}
		(\absorption + \totalscattering) (\dir \cdot \nabla_{\x} v) v  
		\d \x \\
		& \quad -\int_{\inflowelement\backslash\partial\Omega} (\dir \cdot \normal_{\element} ) \ujump{v} \dir \cdot \nabla_{\x} v^+ \d s \\
        & \quad -\int_{\inflowelement\cap\partial\Omega} (\dir \cdot \normal_{\element}) v^+ \dir \cdot \nabla_{\x} v^+ \d s
		\Big)
		\d \dir
		\d \energy \\
		&\equiv {\rm I} + {\rm II} + {\rm III} + {\rm IV}.
	\end{align*}
	Term ${\rm I}$ is already in the required form; employing Lemma~\ref{Lem:Inverse_H1_L2}, Term ${\rm II}$ may be bounded as follows:
	\begin{align*}
	|{\rm II}| &\leq
		\norm{\absorption + \totalscattering}_{L_{\infty}(\domain)} 
		\int_{\energydomain}
		\int_{\angledomain} 
		\sum_{\element \in \spacemesh}
		\tau_{\element}
		\|\dir \cdot \nabla_\x v\|_{L_2(\element)} \|v\|_{L_2(\element)} \d \dir
		\d \energy\\
		&\leq 
		\left(\Cinvele\right)^{\nicefrac{1}{2}}\frac{\norm{\absorption + \totalscattering}_{L_{\infty}(\domain)}}{\poscondmin}
		\| \sqrt{\poscond} v \|^2_{L_2(\domain)} .
	\end{align*}
We now consider the face terms present in terms ${\rm III}$ and ${\rm IV}$; employing the inverse inequality in Lemma~\ref{Lem:Inverse_face_to_ele} together with Young's inequality, we deduce that
		\begin{align*}
		|{\rm III}+{\rm IV}| 
		&\leq 
		\int_{\energydomain}
		\int_{\angledomain} \!
		\sum_{\element \in \spacemesh}\!\!\!
		\big( C_F^2 \Cinvface \big(\| v^+ \!\!- \! v^-\|_{\inflowelement\backslash\partial\Omega}^2
		+
		\| v^+\|_{\partial\element \cap \partial\Omega}^2 \big) \\
		&  \qquad + \frac{\tau_{\element}}{4} \| \dir\cdot\nabla_{\x} v\|_{L_2(\Omega)}^2 
		\big)
		\d \dir
		\d \energy.
	\end{align*}
	
	Finally, we bound the scattering term; recalling 
	the definition of $\totalscattering$ and $\gamma$, employing
	the Cauchy-Schwarz inequality and Lemma~\ref{Lem:Inverse_H1_L2} gives
	\begin{align*}
		& s(v, v_s) =
		\sum_{\element \in \spacemesh}
		\tau_{\element} \\
		&
		\times \int_{\element} 
		\int_{\energydomain}
		\int_{\angledomain} 
		\int_{\energydomain}
		\int_{\angledomain} 
		\scatterkernel(\x, \dir \cdot \otherdir, \otherenergy \to \energy) 
		v(\x, \otherdir,\otherenergy ) \dir \cdot \nabla_\x v (\x, \dir,\energy )
		\d \otherdir 
		\d \otherenergy 
		\d \dir 
		\d \energy 
		\d \x
		\\&
		\leq
		\Big(
		\int_{\spacedomain} 
		\int_{\energydomain}
		\int_{\angledomain} 
		\totalscattering
		v^2
		\d \dir 
		\d \energy 
		\d \x
		\Big)^{1/2} 
		\Big(
		\sum_{\element \in \spacemesh}
		\tau_{\element}^2
		\int_{\element} 
		\int_{\energydomain}
		\int_{\angledomain} 
		\gamma
		(\dir \cdot \nabla_\x v)^2 
		\d \dir 
		\d \energy 
		\d \x
		\Big)^{1/2}
		\\&
		\leq
		(\Cinvele)^{\nicefrac{1}{2}}
		\frac{\norm{\totalscattering}^{\nicefrac{1}{2}}_{L_\infty(\domain)}
		\norm{\gamma}^{\nicefrac{1}{2}}_{L_\infty(\domain)}}{\poscondmin}
		\norm{
		\sqrt{\poscond} v}_{L_2(\domain)}^2.
	\end{align*}
	Combining the individual estimates above, we deduce that
	\begin{align*}
		b(v, w) 
		\geq& 
		\int_{\energydomain}	
		\int_{\angledomain}
		\sum_{\element \in \spacemesh}
		\Big(
		C_1
		 \|\sqrt{\poscond} v \|^2_{L_2(\element)}
		 + C_2
		\Big( \| v^+ - v^-\|_{\inflowelement\backslash\partial\Omega}^2
		+
		\| v^+\|_{\partial\element \cap \partial\Omega}^2 \Big) \\
		&\qquad+ \frac{3\delta}{4} \tau_{\element}
		\norm{\dir \cdot \nabla_\x v}_{L_2(\element)}^2
		\Big)
		\d \dir
		\d \energy.
	\end{align*}
	where
	$$
		C_1 =
		1 - \delta \left(\Cinvele\right)^{\nicefrac{1}{2}}\frac{\norm{\absorption + \totalscattering}_{L_{\infty}(\domain)}}{\poscondmin} 
		- \delta \left(\Cinvele\right)^{\nicefrac{1}{2}}
		\frac{\norm{\totalscattering}^{\nicefrac{1}{2}}_{L_\infty(\domain)}
		\norm{\gamma}^{\nicefrac{1}{2}}_{L_\infty(\domain)}}{\poscondmin},
	$$
	and
	$
	C_2 = 
		\frac{1}{2} - \delta C_F^2 \Cinvface
	$.
	Setting 
		$
		\Lambda_2 
		= 
		\min \Big\{
		\frac{3\delta}{4},
		C_1, C_2
		\Big\}
		$
	which is positive for
	\begin{align*}
		0
		<
		\delta 
		<
		\min \left\{
		\frac{\poscondmin}
		{
		\left(\Cinvele\right)^{\nicefrac{1}{2}} 
		\big( \norm{\absorption + \totalscattering}_{L_\infty(\domain)} 
		      + \norm{\totalscattering}^{\nicefrac{1}{2}}_{L_\infty(\domain)}
		        \norm{\gamma}^{\nicefrac{1}{2}}_{L_\infty(\domain)} 
		\big)
		}
		,
		\frac{1}{2 C_F^2 \Cinvface}
		\right\},
	\end{align*}
	we conclude that
	$
		b(v, w) \geq \Lambda_2 \sdnorm{v}^2
	$
	and the result follows.
\end{proof}

Finally, we state the main result of this paper in the following theorem.

\begin{theorem}[Convergence in the streamline norm]\label{thm:apriori}
Given the mesh $\mesh$ defined over the space-angle-energy domain $\domain$, we assume that the spatial polytopic mesh $\spacemesh$ satisfies Assumptions~\ref{Assumption:mesh_regularity}, \ref{Assumption:bounded_number_of_faces}, and \ref{Assumption:sub-triangulation}. Let $u_h \in \discretespace$ denote the DGFEM approximation satisfying~\eqref{eq:dgScheme}, let  
$u\in H^1(\domain)$ denote the solution of the problem~\eqref{eq:pde} and suppose that $u|_{\kappa} \in H^{l_\kappa}(\kappa)\cup H^1(\element;H^{l_\kappa}(\angleelement\times \energyelement))$, $l_\kappa>1$. 
Then it follows that
\begin{align*}
\sdnorm{u-u_h}^2
\leq& ~C \sum_{\kappa\in\mesh} 
\left( 
	\frac{h_{\element}^{2s_{\element}}}{p_{\element}^{2l_{\kappa}}} 
	\left( {\mathcal L}_\kappa(\absorption,\totalscattering,\gamma )
		+ \frac{1}{h_{\element}^\bot}(1+p_{\element}^2) 
		+ \frac{h_{\element}^\bot}{h^2_{\element}}
	\right) \norm{u}_{H^{l_\kappa}(\kappa)}^2
\right. \\
& \left. 
	+\left(
		\frac{h_{\angleelement}^{2s_{\angleelement}}}{q_{\angleelement}^{2l_{\kappa}}}
		+\frac{h_{\energyelement}^{2s_{\energyelement}}}{r_{\energyelement}^{2l_{\kappa}}}
	\right)
	\left(
		\left( 
		{\mathcal L}_\kappa(\absorption,\totalscattering,\gamma )
		+ \frac{1}{h_{\element}^\bot}
		\right) \norm{u}_{H^{l_\kappa}(\kappa)}^2
	\right.
\right. \\
& \left.
	\left.
		\qquad\qquad\qquad\quad\quad
		+\left(
		\frac{h_{\element}^2}{h_{\element}^\bot} + \frac{h_{\element}^\bot}{p_{\element}^2} 
		\right)
		\norm{u}_{H^1(\element; H^{l_\kappa}(\angleelement\times\energyelement ))}^2
	\right)
\right),
\end{align*}
where
$${\mathcal L}_\kappa(\absorption,\totalscattering,\gamma ) = \norm{\poscond}_{L_\infty(\kappa)} 
		+ (\norm{\absorption+\totalscattering}_{L_\infty(\kappa)}^2
		+\norm{\totalscattering}_{L_\infty(\kappa)}\norm{\gamma}_{L_\infty(\kappa)})\poscondmin^{-1},$$
$s_{\element} = \min (p_{\element}+1,l_{\kappa})$,
$s_{\angleelement} = \min (q_{\angleelement}+1,l_{\kappa})$,
$s_{\energyelement} = \min (r_{\energyelement}+1,l_{\kappa})$ and $C$ is a positive constant
which is independent of the discretization parameters.
\end{theorem}

\begin{proof}
The triangle inequality implies that
\begin{equation}
	\sdnorm{u - u_h} \leq \sdnorm{u - \Pi u} + \sdnorm{\Pi u - u_h}, \label{apriori:eqn1}
\end{equation}
where $\Pi$ denotes the projection operator defined in Section~\ref{Sec:Inverse_and_approx_estimates}. Exploiting
the approximation results derived in Lemma~\ref{Lem:approximation}, the first term on the right-hand side of \eqref{apriori:eqn1} can be bounded as follows:
\begin{align}
	\sdnorm{u - \Pi u}^2
	\leq & ~ 
	C \sum_{\kappa\in\mesh} 
M_{\kappa} \norm{u}_{H^{l_\kappa}(\kappa)}^2
	+T_{\kappa}
		\left(
		\frac{h_{\element}^2}{h_{\element}^\bot} + \frac{h_{\element}^\bot}{p_{\element}^2} 
		\right)
		\norm{u}_{H^1(\element; H^{l_\kappa}(\angleelement\times\energyelement ))}^2
, \label{apriori:eqn2}
\end{align}
where
$$
	T_{\kappa} =
		\frac{h_{\angleelement}^{2s_{\angleelement}}}{q_{\angleelement}^{2l_{\kappa}}}
		+\frac{h_{\energyelement}^{2s_{\energyelement}}}{r_{\energyelement}^{2l_{\kappa}}},
$$
and
$$
	M_{\kappa} = \frac{h_{\element}^{2s_{\element}}}{p_{\element}^{2l_{\kappa}}} 
	\left( \norm{\poscond}_{L_\infty(\kappa)}
		+ \frac{1}{h_{\element}^\bot}(1+p_{\element}^2) 
		+ \frac{h_{\element}^\bot}{h_{\element}^2}
	\right) + T_{\kappa} \left( 
		\norm{\poscond}_{L_\infty(\kappa)}
		+ \frac{1}{h_{\element}^\bot}
		\right).
$$
Recalling the inf-sup bound derived in Theorem~\ref{thm:infsup} and employing Galerkin orthogonality, the second term on the right-hand side of \eqref{apriori:eqn1} can be bounded by
\begin{equation}
	\sdnorm{\Pi u - u_h} \leq \frac{1}{\Lambda} \sup_{w \in \discretespace \setminus \{0\}} \frac{b(u -\Pi u, w)}{\sdnorm{w}}. \label{apriori:eqn3}
\end{equation}
We proceed by estimating the individual terms arising in $b(u - \Pi u, w)$. 
Writing $u_\Pi = u -\Pi u$ and integrating by parts elementwise gives
	\begin{align*}
		& a_{\dir}^{\energy}(u_\Pi, w) \\
		&=  
		\sum_{\element \in \spacemesh} \left(\int_{\element} 
		((\absorption + \totalscattering) u_\Pi w - u_\Pi \dir \cdot \nabla_\x w) \d \x
		\right. \\
		& \qquad\left.
		+\int_{\inflowelement\backslash\partial\Omega} (\dir \cdot \normal_{\element}) \ujump{w} u_\Pi^- \d  s
	    -\int_{\outflowelement\cap\partial\Omega} (\dir \cdot \normal_{\element}) u_\Pi^+ w^+ \d s
	    \right).
\end{align*}
The Cauchy-Schwarz inequality therefore implies that
\begin{align*}
		& a_{\dir}^{\energy}(u_\Pi, w) \\
	    &\leq
	    \sum_{\element \in \spacemesh} 
	    \bigg(
	    \frac{\norm{\absorption + \totalscattering}_{L_\infty(\element)}}{\sqrt{\poscondmin}} 
	    \norm{u_\Pi}_{L_2(\element)} \norm{\sqrt{\poscond} w}_{L_2(\element)}
	    + \norm{u_\Pi^+}_{\outflowelement\cap\partial\Omega} 
	    \norm{w^+}_{\outflowelement\cap\partial\Omega}
	     \\
	    &\qquad
	    +\norm{\tau_{\element}^{-\nicefrac{1}{2}} u_\Pi}_{L_2(\element)}
	    \norm{\tau_{\element}^{\nicefrac{1}{2}} \dir \cdot \nabla_\x w}_{L_2(\element)}
	    + \norm{u_\Pi^-}_{\inflowelement\backslash\partial\Omega} 
	    \norm{w^+-w^-}_{\inflowelement\backslash\partial\Omega}
	    \bigg)
\end{align*}
and, applying the Cauchy-Schwarz inequality once again gives
\begin{align*}
		a_{\dir}^{\energy}(u_\Pi, w)
	    &\leq 
	    \bigg( \!
	    \sum_{\element \in \spacemesh} \!\!\!
	    \bigg( \!\!
	    	\bigg(
	    	\frac{\norm{\absorption + \totalscattering}_{L_\infty(\element)}^2}{\poscondmin} 
	    	+ \frac{1}{\tau_{\element}}
	    	\bigg) \norm{u_\Pi}^2_{L_2(\element)}
	    \\
	    &
	    \qquad
	    \qquad
	    \qquad
	    	\! + \! 2 \norm{u_\Pi^-}^2_{\inflowelement\backslash\partial\Omega}
	    \! + \! 2 \norm{u_\Pi^+}^2_{\outflowelement\cap\partial\Omega}
	    \bigg) \!\!
	    \bigg)^{\frac{1}{2}} \\
	    & 
	    \qquad
	    \qquad
	    \times\bigg(
	    \sum_{\element \in \spacemesh} 
	    \Big(
	    \norm{\sqrt{\poscond} w}_{L_2(\element)}^2 
	    + \tau_{\element}\norm{\dir \cdot \nabla_\x w}^2_{L_2(\element)}
	    \\
	    &
	    \qquad
	    \qquad 
	    \qquad 
	    + \frac{1}{2}\norm{w^+-w^-}_{\inflowelement\backslash\partial\Omega}^2
		+ \frac{1}{2}\norm{w^+}_{\outflowelement\cap\partial\Omega}^2
		\Big) \!\!
	    \bigg)^{\frac{1}{2}}.
	    \end{align*}
	Hence, integrating over energy and angle, and applying the Cauchy-Schwarz inequality and Lemma~\ref{Lem:approximation}, we deduce that
	\begin{align*}
	a(u_\Pi, w) 
		\leq& ~C \left(\sum_{\kappa\in\mesh} 
\left( 
	\frac{h_{\element}^{2s_{\element}}}{p_{\element}^{2l_{\kappa}}} 
	\left( \frac{\norm{\absorption + \totalscattering}_{L_\infty(\kappa)}^2}{\poscondmin}
		+ \frac{1}{h_{\element}^\bot} (1+p_{\element}^2)
	\right) \norm{u}_{H^{l_\kappa}(\kappa)}^2
\right. \right. \\
& \left. 
	+\left(
		\frac{h_{\angleelement}^{2s_{\angleelement}}}{q_{\angleelement}^{2l_{\kappa}}}
		+\frac{h_{\energyelement}^{2s_{\energyelement}}}{r_{\energyelement}^{2l_{\kappa}}}
	\right)
	\left(
		\left( 
		\frac{\norm{\absorption + \totalscattering}_{L_\infty(\kappa)}^2}{\poscondmin}
		+ \frac{1}{h_{\element}^\bot}
		\right) \norm{u}_{H^{l_\kappa}(\kappa)}^2
	\right.
\right. \\
\qquad& \left.
	\left.
	\left.
		\qquad\qquad\qquad\qquad\qquad
		+
		\frac{h_{\element}^2}{h_{\element}^\bot} 
				\norm{u}_{H^1(\element; H^{l_\kappa}(\angleelement\times\energyelement ))}^2
	\right)
\right)
\right)^{\nicefrac{1}{2}} \sdnorm{w}.
	\end{align*}	
	Finally, we consider the scattering term; applying the Cauchy-Schwarz inequality, recalling 
	the definition of $\totalscattering$ and $\gamma$, and using Lemma~\ref{Lem:approximation} gives
	\begin{align*}
		&s(u_\Pi, w)\\
		&\leq
		C 
		\left( 
		\sum_{\kappa\in\mesh} 
		\frac{\norm{\totalscattering}_{L_\infty(\kappa)}
		\norm{\gamma}_{L_\infty(\kappa)}}{\poscondmin}
		\left(
		\frac{h_{\element}^{2s_{\element}}}{p_{\element}^{2l_{\kappa}}}
		+\frac{h_{\angleelement}^{2s_{\angleelement}}}{q_{\angleelement}^{2l_{\kappa}}}
		+\frac{h_{\energyelement}^{2s_{\energyelement}}}{r_{\energyelement}^{2l_{\kappa}}}
		\right) \norm{u}_{H^{l_\kappa}(\kappa)}^2
		\right)^{\nicefrac{1}{2}}
		\norm{\sqrt{\poscond} w}_{L_2(\domain)}.
	\end{align*}
	The result then follows by inserting the above bounds into~\eqref{apriori:eqn3} and using~\eqref{apriori:eqn2}.
\end{proof}

\begin{remark}[$p$-suboptimality of Theorem~\ref{thm:apriori}]\label{remark:convergence}
	Let $h_{\kappa} = \diam(\element)$, $\kappa\in\mesh$, and $h = \max_{\kappa\in\mesh}h_\kappa$, and suppose we have a uniform polynomial degree for all elements, so $p_{\element} = p$ for all $\element\in\spacemesh$, $q_{\angleelement} = p$ for all $\angleelement\in\anglemesh$, $r_{\energyelement} = p$ for all $\energyelement\in\energymesh$.
	Assume that we also have a uniform smoothness degree $s_{\kappa}=s$ for all $\kappa\in\mesh$, $s=\min(p+1,l)$, $l\geq 1$, and that the diameter of the spatial faces of each element $\element\in\spacemesh$ is of comparable size to the diameter of the corresponding element, i.e., so that $h_{\element}^\bot \sim h_{\element}$.
	Then, the {\em a priori} bound stated in Theorem~\ref{thm:apriori} yields
\begin{align*}
	\sdnorm{u-u_h} \sim \mathcal{O} \left(\frac{h^{s-\nicefrac{1}{2}}}{p^{l-1}}\right),
\end{align*}
as $h \rightarrow 0$ and $p\rightarrow \infty$.
This bound is optimal with respect to the meshsize $h$, but suboptimal in the polynomial degree $p$ by half an order, cf. the corresponding result derived in \cite{cangiani2015hp} for the DGFEM approximation of the linear transport problem on (spatial) polytopic meshes.
\end{remark}

\section{Efficient implementation as a multigroup discrete ordinates scheme}
\label{sec:implementation}

The numerical method \eqref{eq:dgScheme} introduced above can be implemented in the framework of a multigroup discrete ordinates scheme.
Although at first sight it appears that the method fully couples the space, angle and energy unknowns, we show that, through a judicious choice of basis functions and element quadrature schemes, it is possible to evaluate the DGFEM solution by simply computing a sequence of linear transport problems in the $d$ spatial variables. To this end, we first consider the multigroup approximation in energy before outlining the angular implementation.

\subsection{Multigroup implementation in energy} \label{sec:multigroup_implementation}

We first show how the energy dependence of the problem may be decoupled.
If we had perfect knowledge of the function
\begin{align*}
	u^+(\x, \dir, \energy) = 
	\begin{cases}
		u(\x, \dir, \energy) \text{ for } \energy > \hat{\energy},
		\\
		0 \text{ otherwise,}
	\end{cases}
\end{align*}
for some $\hat{\energy} > 0$,
then the assumption that the scattering kernel satisfies $\scatterkernel(\x, \otherdir \cdot \dir, \otherenergy \to \energy) = 0$ for $\otherenergy < \energy$, would imply that $\hat{u}(\x, \dir) \equiv u(\x, \dir, \hat{\energy})$ satisfies the monoenergetic radiation transport problem:
find $\hat{u} : \spacedomain \times \angledomain \to \Re$ such that
\begin{align}
	\dir \cdot \nabla_{\x} \hat{u}(\x, \dir) + (\absorption(\x, \dir, \hat{\energy}) + \totalscattering(\x, \dir, \hat{\energy})) \hat{u}(\x, \dir)
	 &= 
	 \scattering[u^+](\x, \dir, \hat{\energy}) \notag \\
	 & \qquad + f(\x, \dir, \hat{\energy}) \text{ in } \domain,
	 \notag
	\\
	\hat{u}(\x, \dir)
	 &= 
	 g(\x, \dir, \hat{\energy}) \text{ on } \Gamma_{\inflow}. \notag
\end{align}
This is the observation underpinning the standard multigroup discretisation: in the discrete setting, we first solve for the fluence in the highest energy group (corresponding to $g = 1$) and then subsequently for each lower energy group in turn. 
Recalling that $\energyelement=(\energy_g,\energy_{g-1})$ denotes the $g$th energy group, $1\leq g\leq N_\energydomain$, we therefore introduce the following family of energy cutoff functions: 
\begin{align*}
	u_g^+(\x, \dir, \energy) = 
	\begin{cases}
		u_h(\x, \dir, \energy) \text{ for } \energy \geq \energy_{g-1},
		\\
		0 \text{ otherwise,}
	\end{cases}
\end{align*}
which represents the component of the discrete fluence which may be considered as pre-computed `data' when solving for the fluence in group $\energyelement$, and focus on solving the problem in a single energy group $\energyelement$, $1\leq g\leq N_\energydomain$.

We expand $u_h$ in group $\energyelement$ in terms of energy basis functions as
\begin{align*}
	u_h(\x, \dir, \energy)|_{\energyelement} \equiv u_g(\x, \dir, \energy) = \sum_{j=1}^{r_{\energyelement} + 1} u_g^j(\x, \dir) \varphi_g^j(\energy),
\end{align*}
where $u_g^j \in \spaceanglespace = \spacespace \otimes \anglespace$, $j=1,2,\ldots,r_{\energyelement} + 1$, and $\{\varphi_g^j\}_{j=1}^{r_{\energyelement}+1}$ forms a basis of $\mathbb{P}_{r_{\energyelement}}(\energyelement)$  (which is only supported on $\energyelement$).
Selecting $v_h=v_g \varphi_g^i \in \discretespace$, with $v_g\in \spaceanglespace$, $i = 1,2, \dots, r_{\energyelement}+1$, the fluence in group $\energyelement$ may then be computed by solving: find $\left\{ u_g^i \right\}_{i=1}^{r_{\energyelement}+1} \in \spaceanglespace$ such that\begin{align}\label{eq:energyDiscrete}
	\sum_{j=1}^{r_{\energyelement}+1}
	\left(
	\int_{\energyelement}
	\int_{\angledomain}
	a_{\dir}^{\energy}(u_g^j, v_g)
	\varphi_g^j
	\varphi_g^i \d \dir
	\d \energy
	-
	s(u_g^j \varphi_g^j, v_g \varphi_g^i) \right)
	=
	s(u_g^+, v_g \varphi_g^i)
	+
	\ell(v_g \varphi_g^i)
\end{align}
for all $v_g \in \spaceanglespace$ and $i = 1,2, \dots, r_{\energyelement}+1$.

Currently, this takes the form of a fully coupled system of monoenergetic Boltzmann transport problems for the $r_{\energyelement}+1$ unknowns within the energy group $\energyelement$.
To simplify this structure, let $\{\energy_g^q\}_{q=1}^{r_{\kappa_g}+1}\subset\kappa_g$ denote the $r_{\kappa_g}+1$ Gauss-Legendre quadrature points on $\kappa_g$ with associated weights $\{\omega_g^q\}_{q=1}^{r_{\kappa_g}+1}\subset\mathbb{R}_{\ge 0}$.
We then select the basis functions $\{\varphi_g^i\}_{i=1}^{r_{\energyelement}+1}$ to be the unique set of polynomials which satisfy the Lagrangian property $\varphi_g^i(\energy_g^j) = \delta_{ij}$, $i,j=1,2, \ldots, r_{\energyelement}+1$, where $\delta_{ij}$ denotes the Kronecker delta. 
This quadrature is exact for polynomials of degree $2 r_{\energyelement}+1$, and so we use it to evaluate the (energy) integrals present in the bilinear form $a_{\dir}^{\energy}(\cdot,\cdot)$, meaning we replace \eqref{eq:energyDiscrete} with:
find $\left\{ u_g^j \right\}_{j=1}^{r_{\energyelement}+1} \in \spaceanglespace$ such that
\begin{align} \label{eq:energyDiscrete2}
	\omega_g^i
	\int_{\angledomain}
	a_{\dir}^{\energy_g^i}(u_g^i, v_g) \d \dir
	-
	\sum_{j=1}^{r_{\energyelement}+1} s(u_g^j \varphi_g^j, v_g \varphi_g^i)
	=
	s(u_g^+, v_g \varphi_g^i)
	+
	\ell(v_g \varphi_g^i)
\end{align}
for all $v_g \in \spaceanglespace$ and $i = 1,2, \dots, r_{\energyelement}+1$. Here, $a_{\dir}^{\energy_g^i}(\cdot,\cdot)$ is defined analogously to $a_{\dir}^{\energy}(\cdot,\cdot)$ with the coefficient data $\alpha$ and $\beta$ evaluated at the energy quadrature point $\energy_g^i$, $i = 1,2, \dots, r_{\energyelement}+1$. Furthermore, with a slight abuse of notation we have written $\left\{ u_g^i \right\}_{i=1}^{r_{\energyelement}+1}$ to also denote the solution of \eqref{eq:energyDiscrete2}, though we stress that \eqref{eq:energyDiscrete2} is an approximation of \eqref{eq:energyDiscrete}\footnote{This quadrature scheme exactly evaluates the integral when the problem data is independent of energy, otherwise it is an approximation which may be expected to be of higher order than the scheme itself when the problem data is sufficiently smooth; see~\cite{Ciarlet:1978}, for example, for a detailed discussion of the role of quadrature in finite element discretisations.}. 

We have not applied the above quadrature scheme in energy to the forcing and scattering terms, since in applications it is usually preferable to treat these terms separately.
Instead, we express the scattering term in an alternative form. 
For $w, v\in \spaceanglespace$, we define
\begin{align*}
	s_{g^{\prime},g}^{j,i}(w,v)	=
	\int_{\angledomain}
	\int_{\spacedomain}
	\int_{\angledomain}
	\Theta_{g^{\prime}, g}^{j,i}(\x, \otherdir \cdot \dir)
	w(\x,\otherdir) v(\x,\dir) \d\otherdir \d\x \d\dir ,
\end{align*}
where
\begin{align*}
	\Theta_{g^{\prime},g}^{j,i}(\x, \otherdir \cdot \dir) = 
	\int_{\kappa_g}
	\int_{\kappa_{g^{\prime}}}
	\scatterkernel(\x, \otherdir \cdot \dir, \otherenergy \to \energy)
	\varphi_{g}^i(\energy)
	\varphi_{g^{\prime}}^j(\otherenergy)
	\d\otherenergy
	\d\energy,
\end{align*}
for $g, g^{\prime} = 1,2,\ldots,N_\energy$, $i=1,2,\ldots,r_{\energyelement}+1$, and $j=1,2,\ldots,r_{\kappa_{g^\prime}}+1$.
With this notation \eqref{eq:energyDiscrete2} may be rewritten in the following equivalent form: find $\left\{ u_g^j \right\}_{j=1}^{r_{\energyelement}+1} \in \spaceanglespace$ satisfying the discrete monoenergetic radiation transport problem
\begin{align} \label{eq:multigroup}
	\omega_g^i
	\int_{\angledomain}
	a_{\dir}^{\energy_g^i}(u_g^j, v_g) \d \dir
	-
	\sum_{j=1}^{r_{\energyelement}+1} s_{g,g}^{j,i}(u_g^j,v_g)
	=
	\sum_{g^\prime=1}^{g-1} \sum_{j=1}^{r_{\kappa_{g^\prime}}+1} s_{g^{\prime},g}^{j,i}(u_{g^\prime}^j,v_g)
	+
	\ell(v_g \varphi_g^i)
\end{align}
for all $v_g \in \spaceanglespace$ and $i = 1,2, \dots, r_{\energyelement}+1$. 
This yields a system of $r_{\energyelement}+1$ monoenergetic radiation transport problems to solve within each energy group, which are only coupled through the scattering operator.
Moreover, the assumed structure of the scattering kernel implies that the problems within a given energy group depend only on the solutions within the same group and from higher energy groups.

\subsection{Discrete ordinates implementation in angle}
We now focus on solving the monoenergetic radiation transport problem~\eqref{eq:multigroup} for a single energy group $g$, $g = 1,2,\ldots,N_\energy$, and energy basis function $\varphi_g^i$, $i=1,2,\ldots,r_{\energyelement}+1$.
To simplify the presentation in this section, we will use $u_h$ to denote $u_g^i$ for an arbitrary $g$ and $i$, and write \eqref{eq:multigroup} in the following simplified form: find $u_h \in \spaceanglespace$ such that
\begin{align}\label{eq:monoenergetic}
	\int_{\angledomain}
	a_{\dir}(u_h, v) \d \dir
	-
	\tilde{s}(u_h, v)
	=
	\tilde{\ell}(v)
\end{align}
for all $v \in \spaceanglespace$, where 
\begin{align*}
a_{\dir}(v,w) &= \omega_g^i a_{\dir}^{\energy_g^i}(v, w), \qquad
\tilde{s}(v, w) = \sum_{j=1}^{r_{\energyelement}+1} s_{g,g}^{j,i}(v,w), \\
 \tilde{\ell}(v) &= \sum_{g^\prime=1}^{g-1} \sum_{j=1}^{r_{\kappa_{g^\prime}}+1} s_{g^{\prime},g}^{j,i}(u_{g^\prime}^j,v)
	+
	\ell(v \varphi_g^i)
\end{align*}
for some (fixed) $g$, $g = 1,2,\ldots,N_\energy$, and some (fixed) $i$, $i=1,2,\ldots,r_{\energyelement}+1$.

For simplicity, we discuss the scheme in the context of the widely-used framework of \emph{source iteration}, although similar simplifications may be incorporated into other linear solvers; indeed, source iteration may be effectively used as a preconditioner within a GMRES solver, for example, see~\cite{radley_thesis_2023}.

We may express the problem~\eqref{eq:monoenergetic} in the following equivalent matrix form: find the vector $U \in \Re^{N}$ of coefficients with respect to a basis of $\spaceanglespace$ such that
\begin{align}\label{eq:dog:dGlinearsystem}
	A U - S U = F
\end{align}
where $A, S \in \Re^{N \times N}$ and $F \in \Re^{N}$ denote the matrix representation of the streaming and scattering operators and load term, respectively.
Source iteration simply refers to the technique of solving this linear system using the Richardson iteration:
given $U^0 \in \Re^{N}$, find $U^r \in \Re^{N}$ such that
\begin{align}\label{eq:dog:sourceIteration}
	A U^r = S U^{r-1} + F,
\end{align}
for $r =1,2, \ldots$.
It may be shown that this iteration converges to the solution of~\eqref{eq:dog:dGlinearsystem} under certain assumptions on the problem data.
The advantage of this approach is that it avoids inverting the scattering matrix, which is typically dense and highly coupled in angle.

To investigate the structure of the matrix $A$, we introduce the following notation: for an angular element $\angleelement$, $\angleelement\in\anglemesh$, we define the local element basis by $\{ \varphi_{\angleelement}^{i}\}_{i=1}^{|q_{\angleelement}|}$, where $|q_{\angleelement}|$ denotes the dimension of the polynomial space defined on $\angleelement$. Furthermore, write $\spacespace = \mbox{span} \{\varphi_{\spacedomain}^{i}\}_{i=1}^{N_{\spacedomain}}$, $N_\spacedomain = \dim(\spacespace)$. Then, noting that the underlying DGFEM does not contain any communication terms between different angular elements, the matrix $A$ has the natural nested block structure 
\begin{align*}
	A = 
	\left[
	    \begin{array}{ccccc}
	    D^1 & 0
	    \\
	    0 & D^2 & 0
	    \\
	    & 0 & \ddots
	    \\
	    && & \ddots & 0
	    \\
	    &&& 0 & D^{|\anglemesh |}
	    \end{array}
	\right],
	\quad\text{ with }\quad
	D^n = 
	\left[
	    \begin{array}{cccc}
	    D^n_{1,1} & \dots & D^n_{1, |q_{\angleelement}|}
	    \\
	    \vdots & \ddots & \vdots
	    \\
	    D^n_{|q_{\angleelement}|,1} & \dots & D^n_{|q_{\angleelement}|,|q_{\angleelement}|}
	    \end{array}
	\right],
\end{align*}
where $|\anglemesh | = \card(\anglemesh )$ and, for $n=1,2,\ldots,|\anglemesh |$, 
$
D^n_{i,j}
=
\int_{\angleelement}
	\varphi_{\angleelement}^i (\dir) \varphi_{\angleelement}^j (\dir)  
	A_{\dir}
\d \dir,
$
$i,j = 1,2,\ldots,|q_{\angleelement}|$,
where
$A_{\dir} \in \Re^{N_{\spacedomain} \times N_{\spacedomain}}$, with $(A_{\dir})_{i,j} = a_{\dir}(\phi_\spacedomain^j , \phi_\spacedomain^i )$, $i,j = 1,2, \ldots, N_{\spacedomain}$.
Solving~\eqref{eq:dog:sourceIteration} therefore requires inverting each diagonal block $D^n$, $n=1,2,\ldots,|\anglemesh |$, which corresponds to solving a coupled system of spatial transport problems on each angular element.

By working once again as in Section~\ref{sec:multigroup_implementation}, this algorithm can be made significantly more efficient. 
To enable this, we restrict the angular mesh to only consist of tensor-product elements, with local element spaces consisting of tensor-product polynomials. 
We can therefore define a basis on each angular element $\angleelement \in \anglemesh$ which satisfies the Lagrangian property with respect to a tensor-product Gauss-Legendre quadrature scheme, simply by using the tensor product of the 1D bases constructed above for the energy discretisation. 
Given the reference element $\refangleelement$, let $\{(\hat{\dir}_q, \hat{\omega}_q)\}_{q=1}^{|q_{\angleelement}|}$ (where $|q_{\angleelement}| = (q_{\angleelement}+1)^{d-1}$) denote the tensor-product Gauss-Legendre quadrature scheme with $q_{\angleelement}+1$ points in each direction. 
As in the 1D case, this scheme exactly integrates polynomials in the space $\mathbb{Q}_{2q_{\angleelement} + 1}(\refangleelement)$. 

On the reference element $\refangleelement$, let $\{\hat{\varphi}_i\}_{i=1}^{|q_{\angleelement}|}$ denote the Lagrangian basis for $\mathbb{Q}_{q_{\angleelement}}(\refangleelement)$ constructed with respect to the Gauss-Legendre quadrature points $\hat{\dir}_q$, $q=1,2,\ldots,|q_{\angleelement}|$, which uniquely satisfies $\hat{\varphi}_i(\hat{\dir}_j) = \delta_{ij}$, $i,j = 1,2,\ldots,|q_{\angleelement}|$. On each angular element $\angleelement$, $\angleelement\in \anglemesh$, we map the local basis defined on the reference element to $\angleelement$ based on employing the mapping $F_{\angleelement}$; more precisely, this yields the local basis $\{\varphi^i_{\angleelement} = \hat{\varphi}_i \circ F_{\angleelement}^{-1} \}_{i=1}^{|q_{\angleelement}|}$ on $\angleelement$. Furthermore, the quadrature scheme on $\angleelement$, $\angleelement\in \anglemesh$, is given by $(\dir_q,\omega_q)_{q=1}^{|q_{\angleelement}|}$, where $\dir_q = F_{\angleelement}(\hat{\dir}_q)$,  $\omega_q = \hat{\omega}_q \jacobian(\hat{\dir}_q)$, $q=1,2,\ldots,|q_{\angleelement}|$, and $\jacobian$ denotes the square root of the determinant of the first fundamental form of the mapping $F_{\angleelement}$. Hence, the mapped basis retains the Lagrangian property of the reference basis.

Using this quadrature to approximate the angular integrals in the first term on the left-hand side of \eqref{eq:monoenergetic}, corresponding to the streaming operator, we deduce that
\begin{align*}
	D^n \approx
	\left[
	    \begin{array}{ccccc}
	    \omega_1
		A_{\dir_1} & 0 &
	    \\
	    0 & \omega_2
		A_{\dir_2} & \ddots
	    \\
	    &\ddots & \ddots &0
	    \\
	    &&0 & \omega_{|q_{\angleelement}|}
		A_{\dir_{N_{|q_{\angleelement}|}}}
	    \end{array}.
	\right]
\end{align*}
Consequently, with this approximation $A$ becomes a block diagonal matrix formed from block diagonal matrices where the individual blocks correspond to a single spatial transport problem.
Solving the source iteration system~\eqref{eq:dog:sourceIteration} therefore only requires the numerical solution of a set of independent spatial transport problems, one for each angular quadrature point, which may be performed in parallel.

\subsection{Full algorithm}
Combining the multigroup energy discretisation and the discrete ordinates angle discretisation described above, we arrive at the efficient algorithm for solving the problem presented in Algorithm~\ref{alg:combined}.
Here, we require a function \texttt{GaussLegendre(}$\omega$\texttt{,}$k+1$\texttt{)} which provides the set of points within the one- or two-dimensional element $\omega$ consisting of $k+1$ points in each dimension, or the mapped analogue for an element on the spherical surface.
The function \texttt{weight} is then used to obtain the quadrature weight associated with a given quadrature point.
The notation \texttt{parfor} indicates a \texttt{for} loop where the individual iterations are independent of one another and may therefore be performed simultaneously and in parallel.

We associate a solution vector $U_{\dir,\energy}$, containing degrees of freedom with respect to the basis $\{\phi^i_\spacedomain\}_{i=1}^{N_{\spacedomain}}$ of $\spacespace$, with each pair of angular quadrature points $\dir$ and energy quadrature points $\energy$ in the natural manner described above. The DGFEM solution $u_h$ is therefore obtained by summing these solution vectors weighted by the space, angle and energy basis functions.

The general structure of the algorithm is to iterate through energy groups in order of decreasing energy, and apply the discrete ordinates algorithm within each group.
We note that the solutions associated with all of the energy basis functions in a given energy group are necessarily coupled together through the scattering operator.
This coupling is quite weak, however, and source iteration reduces this to alternating between two algorithmic steps.
First, the scattering operator is evaluated (using the current solution within the energy group and the previously obtained solution from higher energy groups), which may be performed in parallel.
Second, we solve the spatial transport problem associated with each angle and energy quadrature point.
Again, these are independent problems which may be performed in parallel.

We note that this algorithm could be made more efficient by splitting up the evaluation of the scattering operator into intragroup and intergroup components as in~\eqref{eq:multigroup}, although we do not pursue this here to keep the presentation of the algorithm as simple as possible.

\algrenewcommand\algorithmicindent{1em}%
\begin{algorithm}
\begin{algorithmic}
	\State
	$
	\text{\textbf{inputs:}}
	\begin{cases}
		\text{Energy, angle and space meshes: } &\mesh[\energydomain], \mesh[\angledomain], \mesh[\spacedomain],
		\\
		\text{Polynomial degree vectors: } & {\bf r}, {\bf q}, {\bf p}
		\\
		\text{Source and boundary data: } &f, g,
		\\
		\text{Number of source iterations: } &N \geq 1
	\end{cases}
	$
	\State
	$
	\text{\textbf{initialise }}
		\text{solution vectors } U_{\dir_m,\energy_l}^0 = 0 \in \Re^{N_{\spacedomain}} \text{ for each angle and energy quadrature point }$ \\ $\dir_m \text{ and } \energy_l
	$
	\For{energy group $\energyelement$ with $g \in \{1, \dots, N_{\energydomain}\}$}
		\For{source iteration $t \in \{ 1,\dots,N \}$}
  
			\ParFor{energy quadrature points $\energy_l \in \texttt{GaussLegendre(} \energyelement \texttt{,} r_{\energyelement}+1 \texttt{)}$}
				\ParFor{angular quadrature points $\dir_m \in \bigcup_{\angleelement \in \mesh[\angledomain]} \texttt{GaussLegendre(} \angleelement \texttt{,} q_{\angleelement}+1 \texttt{)}$}

                    \State {Evaluate the scattering operator
					$S_{\dir}^{\energy} \in \Re^{N_{\spacedomain}}$:
                    $$
                        (S_{\dir}^{\energy})_i = s(u_h^{t-1},\varphi_{\spacedomain}^i \varphi_g^l \varphi_{\angleelement}^m)
                    $$
                    }
     
				\EndParFor
			\EndParFor
  
			\ParFor{energy quadrature points $\energy_l \in \texttt{GaussLegendre(}\energyelement \texttt{,} r_{\energyelement}+1 \texttt{)}$}

				\ParFor{angular quadrature points $\dir_m \in \bigcup_{\angleelement \in \mesh[\angledomain]} \texttt{GaussLegendre(} \angleelement \texttt{,} q_{\angleelement} +1 \texttt{)}$}
					\State{
					   Assemble:
                        \begin{align*}
							&\text{transport matrix } A_{\dir}^{\energy} \in \Re^{N_{\spacedomain} \times N_{\spacedomain}} \text{ with } (A_{\dir}^{\energy})_{i,j} = a_{\dir}^{\energy}(\varphi_{\spacedomain}^i , \varphi_{\spacedomain}^j),
							\\
							&\text{source vector } F_{\dir}^{\energy} \in \Re^{N_{\spacedomain}} \text{ with } (F_{\dir}^{\energy})_i = \ell_{\dir}^{\energy}(\varphi_{\spacedomain}^i \varphi_g^l \varphi_{\angleelement}^m).
					   \end{align*}
                    }
					\State{
    					Solve for $U_{\dir,\energy}^{t}$ satisfying:
    					$$
                            A_{\dir}^{\energy} U_{\dir_m,\energy_l}^{t} = \texttt{weight(}\energy\texttt{)}^{-1} \texttt{weight(}\dir\texttt{)}^{-1} (F_{\dir}^{\energy} + S_{\dir}^{\energy})
                        $$
                    }
				\EndParFor
			\EndParFor
		\EndFor
	\EndFor
 
	\State {\textbf{return} angular flux vectors $U_{\dir_m,\energy_l}^{t}$ for each $\dir_m$, $\energy_l$}
\end{algorithmic}
\caption{High order multigroup discrete ordinates implementation of the DGFEM scheme}
\label{alg:combined}
\end{algorithm}

\section{Numerical results}\label{sec:numerics}

In this section we present the results from a series of computational experiments designed to numerically investigate the asymptotic convergence behaviour of the proposed method for both polyenergetic and monoenergetic problems.
The deal.II finite element library in \cite{dealii} was used for the implementation of the method in these numerical examples.

\subsection{Example 1: Polyenergetic problem in 2D}

In this example we consider the numerical approximation of the polyenergetic problem \eqref{eq:pde} posed in a two-dimensional spatial domain, i.e., $d=2$, with a one-dimensional angular domain and a one-dimensional energy domain.
To this end, the spatial domain is defined as $\Omega = (0,1)^2$ (in units of m) and the energy domain is $\energydomain = (500$keV$,1000$keV$)$. 
Furthermore, the macroscopic total absorption cross-section $\alpha$ and the differential scattering cross-section $\theta$ are chosen to mimic Compton scattering of photons travelling through water, see \cite{Davisson:1952}, albeit in a two-dimensional setting.
This is achieved by setting $\alpha = 0$ and
$$ 
\theta(\mathbf{x},\bm{\mu}'\rightarrow\bm{\mu},E'\rightarrow E) = \rho(\mathbf{x}) \sigma_{KN}(E',E,\bm{\mu}\cdot\bm{\mu}') \delta(F(E',E,\bm{\mu}\cdot\bm{\mu}')), 
$$
where $\rho(\mathbf{x}) \approx 3.34281\times10^{29}$e/m$^{3}$ is the electron density of water, and $\sigma_{KN}$ is the Klein-Nishina differential scattering cross-section, see \cite{Davisson:1952}, defined by
$$ 
\sigma_{KN}(E,E',\cos\phi) = \frac12 r_e^2 \left(\frac{E'}{E}\right)^2 \left( \frac{E'}{E} + \frac{E}{E'} - \sin^2\phi \right), 
$$
with $r_e \approx 2.81794\times10^{-15}$m. Further, $\delta$ denotes the Dirac delta distribution and 
$$ 
F(E,E',\cos\phi) = E' - \frac{E}{1+\frac{E}{511}(1-\cos\phi)}, 
$$
is used to enforce the conservation of particle momentum.
Finally, $f$ and $\bc$ are selected so that the analytical solution to \eqref{eq:pde} is given by 
$$ 
u(\mathbf{x},\bm{\mu},E) 
= {\rm e}^{ -\left( \nicefrac{E \bm{\mu}\cdot\mathbf{x}}{E_{max}}\right)^2} \ 
{\rm e}^{ -(1-(\nicefrac{E}{E_{max}})^2)^{-1}}, 
$$
where $E_{max} = 1000$keV.

\begin{figure}[t]
                                \centering
        \begin{tikzpicture}
                \begin{axis}[xmode=log,
                                         ymode=log,
                                         xlabel=$N$,
                                         ylabel=Error, 
                                         width=0.8\textwidth, 
                                         axis background/.style={fill=gray!0}, 
                                         legend pos=south west,
                                         grid=both,
                                         grid style={line width=.1pt, draw=gray!10},
                                         major grid style={line width=.2pt,draw=gray!50}]
                        \addplot+[mark=square, thick, dashed, black, mark options={black, solid}] table [x=n_dofs, y=l2_error, col sep=comma] {numerics/poly_2d_data/l2_errors_0_poly_2d.csv};
                        \addplot+[mark=square*, thick, solid, black, mark options={black}] table [x=n_dofs, y=dg_error, col sep=comma] {numerics/poly_2d_data/dg_errors_0_poly_2d.csv};
                        \addplot+[mark=diamond, thick, dashed, blue, mark options={blue, solid}] table [x=n_dofs, y=l2_error, col sep=comma] {numerics/poly_2d_data/l2_errors_1_poly_2d.csv};
                        \addplot+[mark=diamond*, thick, solid, blue, mark options={blue}] table [x=n_dofs, y=dg_error, col sep=comma] {numerics/poly_2d_data/dg_errors_1_poly_2d.csv};
                        \addplot+[mark=o, thick, dashed, red, mark options={red, solid}] table [x=n_dofs, y=l2_error, col sep=comma] {numerics/poly_2d_data/l2_errors_2_poly_2d.csv};
                        \addplot+[mark=*, thick, solid, red, mark options={red}] table [x=n_dofs, y=dg_error, col sep=comma] {numerics/poly_2d_data/dg_errors_2_poly_2d.csv};
                        
                        \legend{{$p=0$, $L_2(\domain)$-norm}, {$p=0$, DGFEM-norm},
                                        {$p=1$, $L_2(\domain)$-norm}, {$p=1$, DGFEM-norm},
                                        {$p=2$, $L_2(\domain)$-norm}, {$p=2$, DGFEM-norm}};
                                        
                        \addplot[mark=none, solid, black] coordinates {(3e6,5e-2) (3e7,5e-2) (3e7,3.75e-2) (3e6,5e-2)};
                        \addplot[mark=none, solid, black] coordinates {(3e6,9e-3) (3e6,5.06e-3) (3e7,5.06e-3) (3e6,9e-3)};
                        \plot[mark=none] (1e7,5e-2) node[anchor=south] {-0.125}; 
                        \plot[mark=none] (1e7,5.06e-3) node[anchor=north] {-0.25}; 
                        \addplot[mark=none, solid, black] coordinates {(5e7,5e-4) (5e8,5e-4) (5e8,2.11e-4) (5e7,5e-4)};
                        \addplot[mark=none, solid, black] coordinates {(5e7,1e-4) (5e7,3.16e-5) (5e8,3.16e-5) (5e7,1e-4)};
                        \plot[mark=none] (2e8,5e-4) node[anchor=south] {-0.375}; 
                        \plot[mark=none] (2e8,3.16e-5) node[anchor=north] {-0.5}; 
                        \addplot[mark=none, solid, black] coordinates {(2e8,8e-6) (2e9,8e-6) (2e9,1.90e-6) (2e8,8e-6)};
                        \addplot[mark=none, solid, black] coordinates {(2e8,3e-6) (2e8,5.33e-7) (2e9,5.33e-7) (2e8,3e-6)};
                        \plot[mark=none] (6e8,8e-6) node[anchor=south] {-0.625}; 
                        \plot[mark=none] (6e8,5.33e-7) node[anchor=north] {-0.75}; 
                \end{axis}
        \end{tikzpicture}
        \label{fig:polyenergy_2d_errors}

\caption{Example 1: Convergence of the DGFEM under $h$--refinement for $p=0,1,2$. Here, the DGFEM-norm is defined in~\eqref{eq:dGnorm}.} \label{ex1:h-refine}
\end{figure}
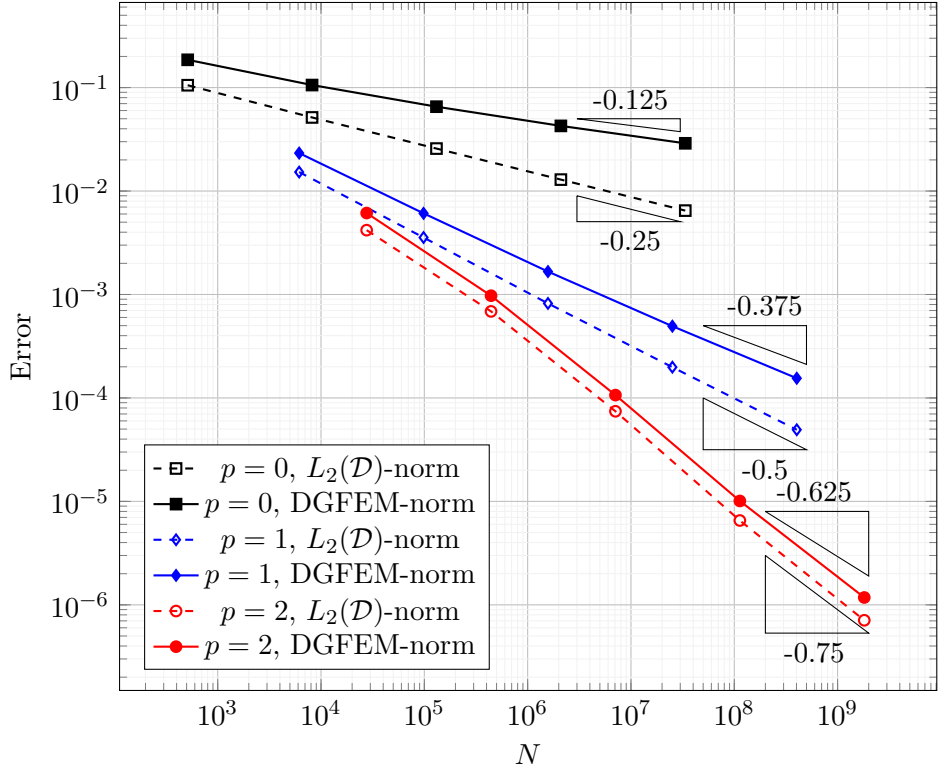

We investigate the asymptotic behaviour of the proposed DGFEM on a sequence of successively finer meshes for different values of the polynomial degrees. 
To this end, the spatial meshes are (non-nested) polygonal grids generated using the Polymesher software package \cite{polymesher}. 
As noted in Section~\ref{sec:angulare_discretisation} the angular meshes are formed by mapping uniform interval elements, defined on the boundary of the square $(-1,1)^2$ to the unit circle $\angledomain$. 
We set polynomial degrees $p_{\element} = p$ for all $\element\in\spacemesh$, $q_{\angleelement} = p$ for all $\angleelement\in\anglemesh$, and $r_{\energyelement} = p$ for all $\energyelement\in\energymesh$.
Figure \ref{ex1:h-refine} shows the error, measured in terms of both the $L_2(\domain)$ and DGFEM-norm, against the number of degrees of freedom (denoted by $N$) in the underlying finite element space $\discretespace$. 
Writing $d_{\domain}$ to denote the dimension of the domain $\domain=\spacedomain\times
\angledomain\times\energydomain$ (here, $d_{\domain}=4$), we clearly observe that $\|u-u_h \|_{L_2(\domain)} \sim {\mathcal O}(N^{\nicefrac{(p+1)}{d_\domain}})$ as the space-angle-energy mesh $\mesh$ is uniformly refined for each fixed $p$. 
Equivalently, since $h \sim N^{-1/d_\domain}$, where $h$ denotes the meshsize of $\mesh$, we note that $\|u-u_h \|_{L_2(\domain)} \sim {\mathcal O}(h^{p+1})$ as $h$ tends to zero for each fixed $p$. 
This is the expected optimal rate of convergence with respect to the $L_2(\domain)$-norm, though this rate of convergence for the DGFEM approximation of first-order hyperbolic PDEs is not guaranteed on general meshes, for further details see \cite{peterson1991note} and the remarks in \cite{cangiani2015hp}. 
Secondly, from Figure \ref{ex1:h-refine} we also observe that for fixed $p$, $p=0,1$, that the DGFEM-norm of the error behaves like ${\mathcal O}(N^{\nicefrac{(p+1/2)}{d_\domain}})$, or equivalently ${\mathcal O}(h^{p+1/2})$, as the meshsize $h$ tends to zero. This is in full agreement with Theorem~\ref{thm:apriori} (see also Remark~\ref{remark:convergence}). In the case when $p=2$, we observe that $\triplenorm{u-u_h}$ converges at a slightly faster rate as $h$ tends to zero; despite the large number of degrees of freedom in $\discretespace$, the meshes are relatively coarse and hence we expect that we are still in the pre-asymptotic regime.

\subsection{Example 2: Monoenergetic problem in 3D}

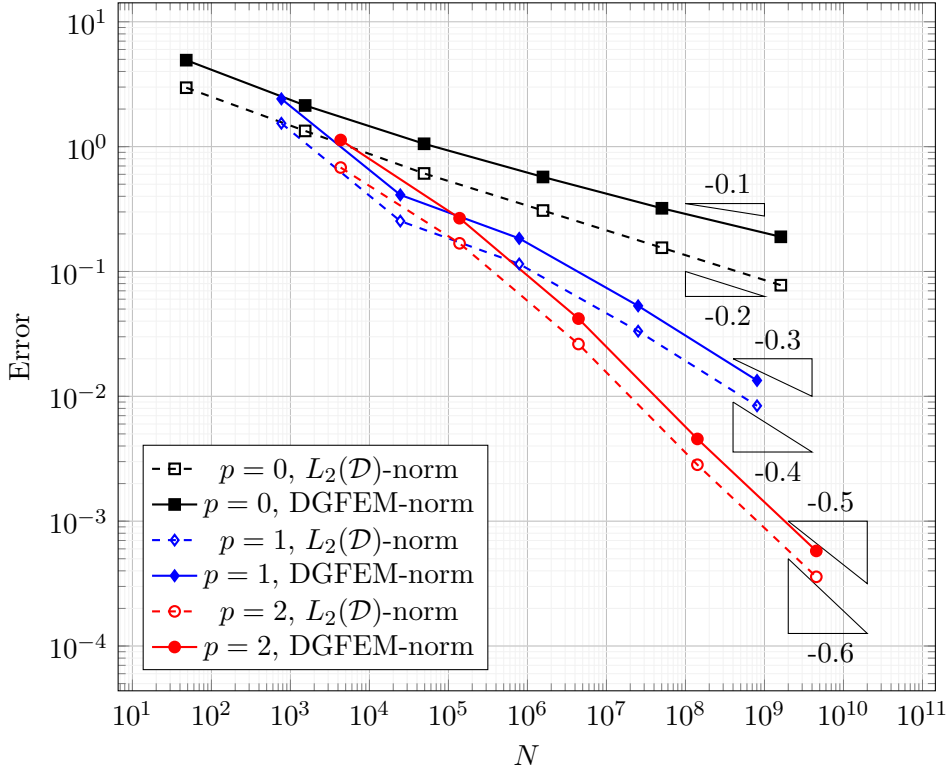
\begin{figure}[t]
                                \centering
        \begin{tikzpicture}
                \begin{axis}[xmode=log,
                                         ymode=log,
                                         xlabel=$N$,
                                         ylabel=Error,
                                         width=0.8\textwidth,
                                         axis background/.style={fill=gray!0},
                                         legend pos=south west,
                                         grid=both,
                                         grid style={line width=.1pt, draw=gray!10},
                                         major grid style={line width=.2pt,draw=gray!50}]
                        \addplot+[mark=square, thick, dashed, black, mark options={black, solid}] table [x=n_dofs, y=l2_space_angle_error, col sep=comma] {numerics/mono_3d_data/errors_0.csv};
                        \addplot+[mark=square*, thick, solid, black, mark options={black}] table [x=n_dofs, y=dg_norm_error, col sep=comma] {numerics/mono_3d_data/errors_0.csv};
                        \addplot+[mark=diamond, thick, dashed, blue, mark options={blue, solid}] table [x=n_dofs, y=l2_space_angle_error, col sep=comma] {numerics/mono_3d_data/errors_1.csv};
                        \addplot+[mark=diamond*, thick, solid, blue, mark options={blue}] table [x=n_dofs, y=dg_norm_error, col sep=comma] {numerics/mono_3d_data/errors_1.csv};
                        \addplot+[mark=o, thick, dashed, red, mark options={red, solid}] table [x=n_dofs, y=l2_space_angle_error, col sep=comma] {numerics/mono_3d_data/errors_2.csv};
                        \addplot+[mark=*, thick, solid, red, mark options={red}] table [x=n_dofs, y=dg_norm_error, col sep=comma] {numerics/mono_3d_data/errors_2.csv};
                        
                        \legend{{$p=0$, $L_2(\domain)$-norm}, {$p=0$, DGFEM-norm},
                                        {$p=1$, $L_2(\domain)$-norm}, {$p=1$, DGFEM-norm},
                                        {$p=2$, $L_2(\domain)$-norm}, {$p=2$, DGFEM-norm}};
                                        
                        \addplot[mark=none, solid, black] coordinates {(1e8,3.5e-1) (1e9,3.5e-1) (1e9,2.78e-1) (1e8,3.5e-1)};
                        \addplot[mark=none, solid, black] coordinates {(1e8,1e-1) (1e8,6.31e-2) (1e9,6.31e-2) (1e8,1e-1)};
                        \plot[mark=none] (3.5e8,3.5e-1) node[anchor=south] {-0.1}; 
                        \plot[mark=none] (3.5e8,6.31e-2) node[anchor=north] {-0.2}; 
                        \addplot[mark=none, solid, black] coordinates {(4e8,2e-2) (4e9,2e-2) (4e9,1.0e-2) (4e8,2e-2)};
                        \addplot[mark=none, solid, black] coordinates {(4e8,9e-3) (4e8,3.58e-3) (4e9,3.58e-3) (4e8,9e-3)};
                        \plot[mark=none] (1.5e9,2e-2) node[anchor=south] {-0.3}; 
                        \plot[mark=none] (1.5e9,3.58e-3) node[anchor=north] {-0.4}; 
                        \addplot[mark=none, solid, black] coordinates {(2e9,1e-3) (2e10,1e-3) (2e10,3.16e-4) (2e9,1e-3)};
                        \addplot[mark=none, solid, black] coordinates {(2e9,5e-4) (2e9,1.26e-4) (2e10,1.26e-4) (2e9,5e-4)};
                        \plot[mark=none] (7e9,1e-3) node[anchor=south] {-0.5}; 
                        \plot[mark=none] (7e9,1.26e-4) node[anchor=north] {-0.6}; 
                \end{axis}
        \end{tikzpicture}
        \label{fig:monoenergy_3d_errors}
\caption{Example 2: Convergence of the method under $h$--refinement for $p=0,1,2$. Here, the DGFEM-norm is defined in~\eqref{eq:dGnorm}.}
\end{figure}

We now consider the numerical approximation of a simplified monoenergetic variant of the problem~\eqref{eq:pde}, where the energy is assumed to remain constant, posed in a three-dimensional spatial domain with a two-dimensional angular domain.
To this end, we let $\spacedomain = (0,1)^3$, $\alpha=1$, $\theta(\mathbf{x},\bm{\mu}'\rightarrow\bm{\mu}) = \nicefrac{1}{|\angledomain^{2}|} = \nicefrac{1}{4\pi}$, $\beta(\mathbf{x}) = \int_\mathbb{S} \theta(\mathbf{x},\bm{\mu}\rightarrow\bm{\mu}') \ d \bm{\mu}' = 1$, and select $f$ and $\bc$ so that the analytical solution of the underlying problem is given by
$$
u(\mathbf{x},\bm{\mu}) = \cos(4\phi)\left( x\cos y+y\sin x \right),
$$
where $\phi = \arccos\bm{\mu}_3$ denotes the polar angle of $\bm{\mu}$.

Figure~\ref{fig:monoenergy_3d_errors} shows the convergence of the DGFEM using meshes comprising of uniform cubes in the spatial domain $\Omega$ and mapped quadrilateral elements in the angular domain $\angledomain$. 
As before, we plot the error measured in both the $L_2(\domain)$-norm and the DGFEM-norm. 
As in the previous example we observe that $\|u-u_h \|_{L_2(\domain)} \sim {\mathcal O}(N^{\nicefrac{(p+1)}{d_\domain}})$, $d_\domain=5$, or equivalently $\|u-u_h \|_{L_2(\domain)} \sim{\mathcal O}(h^{p+1})$ as $h$ tends to zero for each fixed value of the polynomial degree $p$, $p=0,1,2$.
Moreover, we observe that $\triplenorm{u-u_h} \sim {\mathcal O}(N^{\nicefrac{(p+1/2)}{d_\domain}})$ ($\sim {\mathcal O}(h^{p+1/2})$) for $p=0,1$, as $h$ tends to zero.
As in the previous example, we again observe a slighter faster rate of convergence of $\triplenorm{u-u_h}$ for $p=2$, which we attribute to being in the pre-asymptotic regime.

\section{Conclusions}\label{sec:conclusion}

We have introduced a unified $hp$--version DGFEM for the numerical approximation of the linear Boltzmann transport problem. 
We have proven stability and convergence results for the method, through an inf-sup condition in an appropriate norm, and shown how it may be efficiently implemented as a high-order version of the widely used multigroup discrete ordinates method. 
The unified DGFEM formulation in the space, angle and energy domains therefore provides a simple and flexible way of computing arbitrary order approximations of solutions to the Boltzmann transport problem for the first time.
General classes of polytopic elements are admitted for the design of the spatial computational mesh, which facilitates the accurate and efficient representation of complex geometries. 
Numerical experiments have been presented which confirm the theoretical results derived in this paper. 
Further work will include using this scheme within an $hp$-refinement mesh adaptation algorithm, and investigating problems arising in medical physics applications.

\vspace{0.3cm}
\noindent
{\bf Funding} PH and MEH acknowledge the financial support of the EPSRC (grant EP/R030707/1). PH also acknowledges the financial support of the MRC (grant MR/T017988/1). OJS is grateful for the financial support of the UKRI and EPSRC (UKRI Turing AI Fellowship ARaISE EP/V025295/1).

\bibliographystyle{acm}
\bibliography{references}

\end{document}